\documentclass[11pt]{article} 
       
\usepackage[utf8]{inputenc} 
  
\usepackage{geometry} 
\usepackage{graphicx} 
 \usepackage{float}

\usepackage{todonotes}

\usepackage{booktabs} 
\usepackage{array} 
\usepackage{paralist} 
\usepackage{amsmath}
\usepackage{amssymb}
\usepackage{amsthm}
\usepackage{xifthen}
\usepackage{xcolor}
\usepackage{ulem}
\usepackage{hyperref}
\hypersetup{
    colorlinks=true,
    linkcolor=red,
    filecolor=magenta,      
    urlcolor=cyan,
    pdftitle={Overleaf Example},
    pdfpagemode=FullScreen,
}

\usepackage{subcaption}
\usepackage{comment}
\usepackage{soul}

\usepackage[commentmarkup=footnote,commandnameprefix=always]{changes}
\definechangesauthor[name={Andi}, color=red]{and}
\definechangesauthor[name={Kristoffer}, color=blue]{kri}
\definechangesauthor[name={Neo}, color=purple]{neo}

\usepackage{fancyhdr} 
\usepackage{sectsty}
\allsectionsfont{\sffamily\mdseries\upshape} 

\newtheorem{theorem}{Theorem}
\newtheorem*{acknowledgement*}{Acknowledgement}

\newtheorem{corollary}[theorem]{Corollary}

\newtheorem{lemma}[theorem]{Lemma}

\newtheorem{proposition}[theorem]{Proposition}
\newtheorem{remark}[theorem]{Remark}

\usepackage[nottoc,notlof,notlot]{tocbibind} 
\usepackage[titles,subfigure]{tocloft} 

\def\F{\mathcal F}

\def\A{\mathcal A}

\def\R{\mathbb{R}}

\def\N{\mathbb{N}}

\def\Pr{\mathbb{P}}

\renewcommand{\phi}{\varphi} 

\DeclareMathOperator{\E}{\mathbb E}

\newcommand{\kom}[1]{}
\renewcommand{\kom}[1]{{\bf [#1]}}

\usepackage{accents}
\newcommand{\ubar}[1]{\underaccent{\bar}{#1}}
\newcounter{komcounter}
\numberwithin{komcounter}{section}

\renewcommand{\hat}{\widehat} 
\renewcommand{\epsilon}{\varepsilon} 

\renewcommand{\ubar}{\underline}
\renewcommand{\bar}{\overline}

\numberwithin{equation}{section}
\numberwithin{theorem}{section}
\usepackage{changepage}

\title{Outrunning the Omega Clock: A Singular Control Problem for Dividend Optimisation with Ruin and Time-in-Distress Default}
\author{
\\ Andi Bodnariu \\Department of Mathematics, Stockholm University\\
\\Nils Engler\\Department of Mathematics, Stockholm University\\
\\Neofytos Rodosthenous\\Department of Mathematics, University College London
}
 
\begin{document}

\maketitle

\begin{abstract} 
This paper extends the classical dividend problem by incorporating a novel, path-dependent mechanism of firm default. In the traditional framework, ruin occurs when the surplus process first reaches zero. In contrast, default in our model may also arise when the surplus spends an excessive amount of time below a distress threshold, even without ever hitting zero. This occupation-time-based default criterion captures financial distress more realistically, as prolonged periods of low liquidity or capitalisation may trigger regulatory intervention or operational failure.
The resulting optimisation problem is formulated as a new singular stochastic control problem with discontinuous state-dependent discounting and killing. 
We provide a complete analytical solution via a bespoke sequential guess-and-verify method and identify three distinct classes of optimal dividend strategies corresponding to different parameter regimes of the dual-ruin structure. Notably, for certain distress thresholds, the optimal policy features disconnected action and inaction regions.
We further show that, unlike in the classical dividend problem, higher effective discounting induced by occupation time below a distress level can lead to delayed, rather than earlier, dividend payments.
\end{abstract} 
\begin{adjustwidth}{2.5em}{2.5em}

\smallskip


\smallskip
\noindent
{\textbf{MSC2020 subject classification}}: 93E20; 60J60; 49L12; 91B70 
\end{adjustwidth}

\section{Introduction}

The classical dividend problem is a cornerstone of actuarial science and financial risk theory, addressing the optimal strategy for distributing dividends from a firm’s surplus while balancing profitability against the risk of ruin. 
In its most basic form, the surplus is modelled as a stochastic process -- typically a Brownian motion with drift or a compound Poisson process -- and ruin is defined as the first hitting time of zero. 
The problem was first introduced by de Finetti \cite{de1957impostazione}, who studied the maximisation of the expected cumulative discounted dividends paid until ruin. 
Since then, it has been extensively investigated within the framework of stochastic singular control, leading to a rich and well-established theory. 
Early contributions include \cite{shreve1984optimal}, which formulates the problem as a singular control problem, and \cite{jeanblanc1995optimization}, which allows for both continuous dividend rates and lump-sum payments. 
For comprehensive surveys, we refer the reader to \cite{albrecher2009optimality} and \cite{schmidli2007stochastic}.

However, real-world corporate default rarely manifests as an abrupt event occurring precisely at zero capital. 
Regulatory pressure, liquidity stress, and the erosion of market confidence often force firms into default or restructuring while they remain technically solvent. 
In practice, a prolonged stay in a low-surplus region -- even if not resulting in classical ruin -- can severely impair operations or trigger pre-emptive regulatory or managerial intervention. 
Motivated by these considerations, we propose a novel extension of the classical dividend optimisation problem in which the surplus process is subject to two distinct modes of default: traditional ruin at zero and an additional mechanism based on the cumulative occupation time spent in a low-surplus region, leading to what we term {\it occupation-time-induced default}.

The firm’s surplus process $X^D$ is modelled as a controlled Brownian motion with drift,
$X^D_t = x + \mu t + \sigma W_t - D_t$ for all $t \geq 0$, where $D$ denotes the cumulative dividend process. 
To incorporate the additional default mechanism described above, we introduce an omega-clock framework, under which the firm may default if the cumulative time spent by the surplus in a low-surplus region exceeds an independent exponentially distributed random time. 
Specifically, we define the distress region as the interval $[0,y]$ representing the low-surplus zone below the distress threshold $y>0$, where the firm is considered to be in financial stress, in contrast to the operationally safe region above $y$. 
The omega clock is then given by the occupation time
\begin{align} \label{omega_clock}
\omega_t^y = q \int_0^t I_{\{X_s^D<y\}}ds, \quad t \geq 0,
\end{align}
where $q>0$ governs the rate at which time spent in distress accumulates toward a potential default, which occurs when $\omega_t^y$ exceeds an exponentially distributed random threshold. 
This construction introduces a path-dependent stochastic killing mechanism that captures persistence-based financial distress. 
Consequently, management faces the dual challenge of optimising dividend payments while avoiding both instantaneous ruin at zero and accumulated distress-induced default. 
The resulting optimisation problem is formulated as a singular stochastic control problem with a random, path-dependent horizon determined jointly by classical ruin and the omega-clock mechanism.

The introduction of this occupation-time-based default mechanism within the dividend optimisation framework -- bridging instantaneous ruin and prolonged underperformance -- constitutes the first contribution of this paper. 
To the best of our knowledge, existing dividend problems with ruin at zero and random time horizons are restricted to settings in which the horizon is independent of the surplus process (see, e.g.~\cite{doi:10.1142/9789814383318_0007}, \cite{zhao2015optimal}). 
Moreover, we show that the problem can be equivalently reformulated as a singular stochastic control problem with a discontinuous, stochastic, state-dependent instantaneous discount rate, which is also novel in this context. 
Our second main contribution is the derivation of closed-form solutions exhibiting a surprisingly rich qualitative structure, with optimal dividend strategies changing markedly across different distress threshold regimes. 
These results provide a tractable framework for further theoretical developments and offer practical insight into dividend strategy design under time-based regulatory or financial pressure, with potential applications in finance and insurance.

The use of occupation times in stochastic models dates back to \cite{chesney1997brownian}, which introduced Parisian barrier options. 
Since then, occupation times and Parisian-type ruin -- where ruin occurs if the surplus process remains below a given level for a sufficiently long excursion -- have also been studied in the context of dividend optimisation (see, e.g.~\cite{czarna2014dividend}, \cite{xu2022optimal}, \cite{wang2024optimality}). 
A key distinction between Parisian ruin and the present framework lies in the memory structure of the associated timer: 
under Parisian ruin, the clock is reset whenever the surplus recovers above the threshold, even if only briefly, whereas the omega clock \eqref{omega_clock} employed here accumulates the total time spent in distress, without erasing past distress periods after short recoveries. 
An early contribution using the omega-clock mechanism is \cite{albrecher2011optimal}, which considers the special case of a distress threshold $y=0$ and excludes immediate killing at zero. 
In contrast, we allow for general distress thresholds $y\in[0,\infty)$ -- under which the optimal strategy is shown to change substantially -- and incorporate classical ruin at zero as an additional killing mechanism. 
A related occupation-time-based killing mechanism induced by an omega clock has also been studied in the context of optimal stopping problems in \cite{neo_omega2018}.

Our problem is also related to surplus-based models inspired by Chapter~7 and Chapter~11 proceedings under the U.S. Bankruptcy Code. 
An early contribution in this direction is \cite{Chapter11_early}, which proposes a framework in which firms transition between liquid and distressed regimes and may face liquidation or reorganisation depending on both surplus levels and the persistence of financial distress. 
Such models can be characterised by a triplet of thresholds $a<b<c$: 
when the surplus falls below $b$, the firm enters a distress regime; it may undergo immediate liquidation if the surplus drops below $a$; default may occur if the firm remains in distress for a sufficiently long time; or the firm may return to a liquid phase if the surplus exceeds $c$. 
A first attempt to study optimal dividend policies in this spirit is given in \cite{wang2024optimality}, which focuses on Chapter~11–type default, triggered by prolonged distress but excluding immediate liquidation (i.e.~$a=-\infty$). 
In contrast, our model incorporates classical ruin at zero (analogous to Chapter~7) as an additional default mechanism alongside the occupation-time-based default.

From a methodological perspective, we aim to solve the associated singular stochastic control problem with discontinuous, stochastic, state-dependent instantaneous discount rate. 
A special case of discontinuous stochastic discounting was considered in \cite{albrecher2011optimal}, corresponding to a distress level $y=0$ and without immediate ruin at zero. 
In that setting, the standard guess-and-verify methodology suffices, yielding the classical solution of reflecting the surplus downward at a single upper boundary. 
In contrast, we show that this approach cannot provide a solution for general distress levels $y>0$. 
Moreover, due to the discontinuity of the discount rate, the standard connection between singular control and optimal stopping problems (see, e.g.~\cite{karatzas1984connections}, \cite{karatzas1985}) does not hold, precluding its use for solving the problem. 
As a result, we develop a different approach, demonstrating that, for certain parameter regimes, the optimal control exhibits two disjoint inaction (waiting) regions and two disjoint action (dividend-paying) regions, separated by three free boundaries. 
Consequently, the optimal policy may involve, in addition to the standard initial lump-sum dividend, a subsequent lump-sum payment when the surplus enters a lower action region from above.

Multi-barrier strategies have previously appeared in the Cramér–Lundberg model (see, e.g.~\cite{albrecher2023optimal}, \cite{azcue2005optimal}) and in singular control problems such as \cite{junca2024optimal} (state-dependent rewards with constant discounting) and \cite{bai2012non} (dividends with proportional and transaction costs, leading to impulse-type controls). 
To the best of our knowledge, the occurrence of multi-barrier strategies in the context of stochastic discounting or random time horizons -- including models where the discounting is an exogenous stochastic process (see, e.g.~\cite{AKYILDIRIM201493}, \cite{https://doi.org/10.1111/mafi.12339}) -- has not been reported previously. 
This makes our work the first to identify such a solution structure in this setting.

This structural richness highlights the nuanced effects of occupation-time-based risk, which cannot be captured by simpler models of ruin or path-independent discounting. 
A key feature of our model is that the solution to the optimal control problem depends critically on the distress threshold $y$, which partitions the surplus state space into a distress region $[0,y]$ and a no-distress region $(y,\infty)$. 
We identify three qualitatively distinct regimes based on the size of the distress region; see Figure \ref{fig:1} for a numerical illustration.
In the {\it subcritical regime} (small $y$), distress is triggered only at very low surplus levels, and the optimal dividend strategy takes the form of a classical Skorokhod reflection at a boundary that depends on $y$. 
In the {\it supercritical regime} (large $y$), the firm behaves as if it is effectively always under distress, and the optimal strategy coincides with that of the fully penalised case (i.e.~$y=\infty$), becoming independent of $y$. 
The most interesting and novel behaviour arises in the {\it critical regime} (intermediate $y$), where the optimal strategy exhibits a genuinely new structure, with disconnected action and inaction regions emerging from the interaction between occupation-time penalisation, ruin at zero, and surplus dynamics. 
This regime represents a qualitative departure from classical results in singular control and risk theory.

The paper is organised as follows. 
Section \ref{sec_math_model} introduces the model and formal problem setup, and presents a general verification theorem applicable for all subsequent regimes. 
In Section \ref{chapter_classical_case}, we recall the well known results of the classical dividend problem with ruin at zero surplus, and in Section \ref{sec:bounds} we obtain robust bounds for the value function. 
In Section \ref{sec:results} we present the solution to the problem considering three distinct parameter regimes. 
In particular, in Sections \ref{sec:supercritical}--\ref{sec_subcritical}, we consider the supercritical (resp., subcritical) regime with high (resp., low) distress level $y$, resulting in an optimal control policy given by Skorokhod reflection. 
In Section \ref{sec_critical_regime}, we consider the critical regime  with intermediate distress levels $y$, where the optimal policy is characterised by two disjoint payout regions, and two disjoint waiting regions.
The construction of optimal controls, omitted technical proofs, and auxiliary results can be found in Apendices \ref{sec:construction}--\ref{sec:auxiliary}.

\section{Mathematical Model \& Preliminaries}
\label{sec_math_model}

Let  $(\Omega, \F, \left({\cal F}_t\right)_{t\geq 0}, \mathbb{P})$ be a filtered probability space satisfying the usual conditions, supporting a standard one-dimensional Brownian motion $W=\left(W_t\right)_{t\geq0}$, whose augmented natural filtration is $\left({\cal F}_t\right)_{t\geq 0}$
and an independent exponential random variable $e_1$ with unit mean (in particular, independent of ${\cal F}_\infty \subset \mathcal{F}$).

We define the uncontrolled surplus process $X^0=(X_t^0)_{t\geq 0}$ (without any interventions) by the following Brownian motion with drift: 
\begin{align} \label{uncontrol_SDE}
dX_t^0=\mu dt+ \sigma dW_t, \quad t\geq 0, \quad X^0_{0-} = x\geq 0, 
\end{align}
where $\mu>0$ represents the positive drift of the firm's surplus, $\sigma>0$ represents its volatility coefficient and $x>0$ is the initial positive surplus level of the firm. 
We denote by $\mathcal{L}$ the infinitesimal generator of the uncontrolled process $X^0$ satisfying \eqref{uncontrol_SDE}, which is defined at least for functions $f \in {\cal C}^2(\R)$ by
\begin{align*}
\mathcal{L} f(x) := \mu f'(x) + \frac{1}{2} \sigma^2 f''(x), 
\quad \text{for } x \in \R.
\end{align*}

The corresponding controlled surplus process $X^D=(X^D_t)_{t\geq 0}$ is given by 
\begin{align} \label{control_SDE} 
dX_t^D = \mu dt+ \sigma dW_t -dD_t, \quad t\geq 0, \quad X^D_{0-}=x\geq 0, 
\end{align}
where the control process $(D_t)_{t\geq0}$ represents the cumulative amount of dividends paid from time $0$ up to time $t$ and belongs to the set of admissible control processes $\mathcal{A}$ defined by 
\begin{align} \label{adm}
\begin{split}
\mathcal{A} := \{ (D_t(\omega))_{t\geq0} &\text{ is increasing, right-continuous, $\left({\cal F}_t\right)$-adapted,}\\
&\; \text{ with $D_{0-}=0$, such that  $D_{t}-D_{t-} \leq X_{t-}^D$, $\mathbb{P}$-a.s., $\forall$ $t \geq 0$\}}.
\end{split}
\end{align}

The condition that dividend strategies are increasing and adapted means that paid-out dividends cannot be returned and the dividend strategy can be constructed using only the uncontrolled paths given by $X^0$. 
Furthermore, the condition $D_{t}-D_{t-} \leq X_{t-}^D$ ensures that we cannot pay out more dividends than the surplus level, resulting in negative surplus.
For any fixed $y \in [0,\infty)$, we are interested in the optimal dividend problem corresponding to 
\begin{align}\label{exp_reward_value}
V(x,y) := \sup_{D\in\mathcal{A}} J(x;y,D), 
\quad \text{where} \quad 
J(x;y,D) := \E_x\bigg[\int_0^{\tau^D_0}e^{-rt}I_{\{\omega^y_t<e_1\}}dD_t\bigg], 
\quad x \in \R,
\end{align}
where $\omega^y$ is the omega clock defined by \eqref{omega_clock}, $I$ is the indicator function, the traditional time of ruin $\tau^D_0$ at zero surplus is defined by 
$$
\tau^D_0 := \inf\big\{t\geq 0 : X^D_t\leq 0 \big\}
$$ 
and $\mathcal{A}$ is the set of all admissible controls defined by \eqref{adm}.
Henceforth we will refer to $V(\cdot;y)$ as the value function and $J(\cdot;y,D)$ as the expected reward under the dividend policy $D$ and the distress threshold $y \geq 0$. 

Before proceeding to the analysis of the problem \eqref{exp_reward_value} with random time-horizon modelled by the omega clock, we note that it can be formulated into a problem with path-dependent discounting, given in terms of a state-dependent discount rate. 
This is presented in the following result. 

\begin{lemma}
\label{discount_chance_lemma}
Suppose that $y \in [0,\infty)$ and $J(\cdot;y,D)$ is the expected reward defined in \eqref{exp_reward_value} for any $D \in \mathcal{A}$. 
Then, we have 
\begin{align}
\label{expected_reward_discount}
J(x;y,D)=\E_x\bigg[\int_0^{\tau^D_0}e^{-rt-\omega^y_t}dD_t\bigg], 
\quad x \in \R.
\end{align}
\end{lemma}

\begin{proof}
For any fixed $D \in \mathcal{A}$, we can use the tower property, Fubini theorem and the independence of $e_1$ and $X$ (also $\mathcal{F}_{\infty}$), to obtain from \eqref{exp_reward_value} that
\begin{align*}
J(x;y,D) 
= \E_x\bigg[ \int_0^{\tau^D_0}e^{-rt}I_{\{\omega^y_t<e_1\}}dD_t \bigg] 
&= \E_x\bigg[ \E_x\bigg[ \int_0^{\tau^D_0} e^{-rt} I_{\{\omega^y_t<e_1\}} dD_t \bigg| \mathcal{F}_{\infty} \bigg] \bigg] \\
&= \E_x\bigg[ \int_0^{\tau^D_0} e^{-rt} \Pr_x \left( \omega^y_t<e_1 | \mathcal{F}_{\infty} \right) dD_t \bigg] 
= \E_x\bigg[ \int_0^{\tau^D_0} e^{-rt-\omega^y_t} dD_t \bigg],
\end{align*}
for all $x \in \R$, which completes the proof.  
\end{proof}

In the sequel, we focus on solving the singular control problem \eqref{exp_reward_value} when the expected rewards is represented by \eqref{expected_reward_discount}, 
namely, we aim at solving 
\begin{align}\label{vf}
V(x;y) = \sup_{D\in\mathcal{A}} J(x;y,D), \quad \text{where} \quad 
J(x;y,D) = \E_x\bigg[\int_0^{\tau^D_0}e^{-rt-\omega^y_t}dD_t\bigg],
\quad x \in \R.
\end{align}
To that end, we present in what follows sufficient conditions for the optimality of an admissible strategy in the form of a verification theorem.

\begin{theorem} \label{verification_theorem}
Suppose that $y \in \mathcal{I} \subseteq [0,\infty)$ 
and $D^*\in\mathcal{A}$ is an admissible control process such that the expected reward $J(\cdot;y,D^*)$ given by \eqref{vf} satisfies the regularity conditions $J(\cdot;y,D^*) \in \mathcal{C}^2((0,b) \cup (b,y) \cup (y,\infty)) \cap \mathcal{C}^1(0,\infty)$ for some $0 \leq b \leq y$. 
If $J(\cdot;y,D^*)$ further satisfies the conditions
\begin{align}
J'(x;y,D^*) &\geq 1, {\quad x\in (0,\infty),} \label{mainthmcond1}  \tag{I} \\
\mathcal{L} J(x;y,D^*) \, I_{\{x\not\in\{b,y\}\}} -\big(r + q \, I_{\{x<y\}} \big) \, J(x;y,D^*) &\leq 0, {\quad x\in (0,\infty),} \label{mainthmcond2} \tag{II}
\end{align}
then $J(x;y,D^*) = V(x;y)$ in \eqref{vf} and the control $D^*$ is optimal.
\end{theorem}
\begin{proof}
Let $D \in\mathcal{A}$ be an arbitrary admissible control process and denote its jumps by $\Delta D_t:=D_{t}-D_{t-}$ and its continuous part by $D^c_t:=D_t-\sum_{0\leq s\leq t}\Delta D_s$. 
We also introduce the stopping time
\begin{align*}
\tau_n=n\wedge\inf\{t\geq 0:X^D_t\not\in (0,n)\}.
\end{align*}
and then by applying a generalised It\^{o}'s formula (cf.~e.g.~\cite[p.74]{Peskir} or \cite[Theorem 3.2.]{peskir2007change}), we get 
\begin{align*}
&\exp\Big\{-\int_0^{\tau_n\wedge\tau_0^D} (r + qI_{\{X^D_s<y\}})ds \Big\} J(X^D_{\tau_n\wedge\tau_0^D};y,D^*)-J(x;y,D^*)\\
&\quad= \int_0^{\tau_n\wedge\tau_0^D}\exp\Big\{-\int_0^{t} (r + qI_{\{X^D_s<y\}})ds \Big\}\left(\mathcal{L}J(X^D_t;y,D^*) I_{\{x\not\in\{b,y\}\}} - (r+qI_{\{X^D_t<y\}}) J(X^D_t;y,D^*) \right) dt \\
&\quad\quad + \int_0^{-\tau_n\wedge\tau_0^D} \exp\Big\{-\int_0^{t} (r + qI_{\{X^D_s<y\}})ds \Big\} J'(X^D_t;y,D^*) \big( \sigma dW_t - dD^c_t \big)\\
&\quad\quad - \sum_{0\leq t\leq \tau_n\wedge\tau_0^D} \exp\Big\{-\int_0^{t} (r + qI_{\{X^D_s<y\}}) ds \Big\} \big( J(X_{t-}^D;y,D^*) - J(X_{t-}^D-\Delta D_t;y,D^*) \big),
\end{align*}
since $J(\cdot;y,D^*)$ satisfies $J(\cdot;y,D^*) \in \mathcal{C}^2((0,b) \cup (b,y) \cup (y,\infty)) \cap \mathcal{C}^1(0,\infty)$.  
Thus, by taking expectations, using that the stochastic integral is a martingale, as well as the fundamental theorem of calculus, we get after rearrangements that
\begin{align*}
&J(x;y,D^*) 
= \E_x\bigg[\exp\Big\{-\int_0^{\tau_n\wedge\tau_0^D} (r + qI_{\{X^D_s<y\}})ds \Big\} J(X^D_{\tau_n\wedge\tau_0^D};y,D^*)\bigg] \\ 
&\quad - \E_x\bigg[\int_0^{\tau_n\wedge\tau_0^D}\exp\Big\{-\int_0^{t} (r + qI_{\{X^D_s<y\}})ds \Big\} \left(\mathcal{L}J(X^D_t;y,D^*) I_{\{x\not\in\{b,y\}\}} - (r+qI_{\{X^D_t<y\}}) J(X^D_t;y,D^*) \right) dt \bigg] \\ 
&\quad + \E_x\bigg[\int_0^{\tau_n\wedge\tau_0^D} \exp\Big\{-\int_0^{t} (r + qI_{\{X^D_s<y\}})ds \Big\} J'(X^D_t;y,D^*) dD^c_t \bigg] \\ 
&\quad + \E_x\bigg[\sum_{0\leq t\leq \tau_n\wedge\tau_0} \exp\Big\{-\int_0^{t} (r + qI_{\{X^D_s<y\}}) ds \Big\} \int_0^{\Delta D_t} J'(X^D_{t-}-z;y,D^*) dz \bigg].
\end{align*}
Therefore, by using the conditions \eqref{mainthmcond1}--\eqref{mainthmcond2} for $J(\cdot;y,D^*)$, we get
\begin{align*}
J(x;y,D^*) &\geq \E_x\bigg[\int_0^{\tau_n\wedge\tau_0^{D^*}} \exp\Big\{-\int_0^{t} (r + qI_{\{X^{D^*}_s<y\}})ds \Big\} dD^c_t + \sum_{0\leq t\leq \tau_n\wedge\tau_0} \hspace{-2mm}\exp\Big\{-\int_0^{t} (r + qI_{\{X^D_s<y\}})ds \Big\} \Delta D_t \bigg] \\ 
&= \E_x\bigg[\int_0^{\tau_n\wedge\tau_0^D} \exp\Big\{-\int_0^{t} (r + qI_{\{X^D_s<y\}})ds \Big\} dD_t \bigg].
\end{align*}
Taking the limits as $n \to \infty$ and an application of the monotone convergence theorem, yields that $J(x;y,D^*) \geq J(x;y,D)$. 
Since $D \in \A$ is arbitrary, we can take the supremum on the right-hand side over all $D \in \A$, which gives thanks to the definition \eqref{vf} of $V$ that $J(x;y,D^*) \geq V(x;y)$. 

Since $J(\cdot;y,D^*)$ is also given by \eqref{vf} and $D^*\in\mathcal{A}$ by assumption, we can therefore conclude that $J(x;y,D^*) = V(x;y)$ and consequently that $D^*$ is an optimal control.
\end{proof}

This verification theorem is a tailored result that is constructed to be general enough, in order to cover all subsequent structures of our candidates for the optimal control strategy. 
However, this generality needs to be compensated in the forthcoming analysis with additional results on the construction of the control processes and their admissibility, as well as the construction of candidate value functions and their regularity properties.
We do this on a case-by-case basis for each regime in Section \ref{sec:results}.

\subsection{Classical dividend problem with ruin at zero surplus} 
\label{chapter_classical_case}

In this section, we review the classical dividend problem which aims at maximising the expected reward 
\begin{align} \label{reward_classical}
J_\rho(x;D) := \E_x\bigg[ \int_0^{\tau^D_0} e^{-\rho t} dD_t \bigg],
\quad x \in \R,
\end{align}
for a fixed discounting $\rho>0$, whose value function is thus defined by 
\begin{align}\label{classical_value}
V_\rho(x):=\sup_{D \in \mathcal{A}} J_\rho(x;D).
\end{align}
This problem is nowadays well-understood and its solution can be found in \cite{jeanblanc1995optimization,lokka2008optimal}, among others.
The optimal control $D^*$ is can be characterised by
\begin{equation} \label{SR}
D^{b^*_\rho}_t := (x-b^*_\rho)^+ + L^{b^*_\rho}_t(X^{D^{b^*_\rho}}), \quad t\geq 0, \quad D^{b^*_\rho}_{0-}=0 ,  
\end{equation}
where $L^{b^*_\rho}(X^{D^{b^*_\rho}})$ denotes the (symmetric) local time of $X^{D^{b^*_\rho}}$ at the point $b^*_\rho$,
which splits the state-space into a waiting region $\mathcal{W} = (0,b^*_\rho)$, where no control is exerted, and an action region $\mathcal{D} = [b^*_\rho, \infty)$,
which induces an initial jump of the control process to bring the controlled process from $x > b^*_\rho$ to the boundary $b^*_\rho$ of the closure $\overline{\mathcal{W}}$ of the waiting region (if needed), 
and then prescribes that the minimal amount of control is exerted to keep the controlled process in $\overline{\mathcal{W}}$ (thus solving a Skorokhod reflection problem; see Section \ref{scorohod_sec} for details).    

In particular, the boundary $b_\rho^*$ and the value function $V_\rho$ 
satisfy the free-boundary problem (FBP)
\begin{align}
&\tfrac12 \sigma^2 V_\rho''(x)+\mu V_\rho'(x) - \rho V_\rho(x) = 0 , \quad x\in (0,b_\rho^*), \label{FBP_single_tresholda}\\
&V_\rho(x) = V_\rho(b_\rho^*) + x- b_\rho^*, \qquad \qquad \quad \, x \in [b_\rho^*,\infty), \label{FBP_single_tresholdb}\\ 
&V_\rho(0+) = 0,  \label{FBP_single_tresholdc}\\ 
&V_\rho \in {\mathcal{C}^2(0,\infty)}
\label{FBP_single_tresholdc}.
\end{align}
and take the form 
\begin{align} \label{value_classical}
V_\rho(x) = \begin{cases}
\dfrac{e^{\gamma_1(\rho) x} - e^{\gamma_2(\rho)x}}{\gamma_1(\rho) e^{\gamma_1(\rho) b^*_\rho} - \gamma_2(\rho) e^{\gamma_2(\rho)b^*_\rho}}, 
\quad &x\in[0,b^*_\rho),\\
V_\rho(b^*_\rho) + x - b^*_\rho, 
\quad &x\in [b^*_\rho,\infty),
\end{cases}
\quad \text{and} \quad 
b_\rho^* = \frac{\log\left(\gamma_2^2(\rho) / \gamma_1^2(\rho)\right)}{\gamma_1(\rho)-\gamma_2(\rho)} ,
\end{align}
where $\gamma_2(\rho) < 0 < \gamma_1(\rho)$, for all $\rho>0$, and are given by 
\begin{align} \label{gamma_def}
\gamma_1(\rho)= \sqrt{\frac{\mu^2}{\sigma^4}+\frac{2\rho}{\sigma^2}}-\frac{\mu}{\sigma^2},\quad \gamma_2(\rho)= -\sqrt{\frac{\mu^2}{\sigma^4}+\frac{2\rho}{\sigma^2}}-\frac{\mu}{\sigma^2}
\end{align}
and by using the facts that 
\begin{align} \label{lemma_useful_results_gamma} 
\gamma_1(\rho+q) + \gamma_2(\rho + q)
= \gamma_1(\rho) + \gamma_2(\rho)
= \tfrac{\mu}{\rho}  \gamma_1(\rho) \gamma_2(\rho), 
\quad q > 0,
\end{align}
we further have that 
\begin{align} \label{lemma_useful_results_Vb}
V_{\rho}(b^*_{\rho})
=\frac{\gamma_2^2(\rho)-\gamma_1^2(\rho)}{\gamma_2^2(\rho)\gamma_1(\rho)-\gamma_1^2(\rho)\gamma_2(\rho)}
= \frac{\gamma_1(\rho) + \gamma_2(\rho)}{\gamma_1(\rho) \gamma_2(\rho)} = \frac{\mu}{\rho}.
\end{align}
In particular $V_\rho(x)=J_\rho(x;D^{b^*})$ also satisfies the conditions of Theorem \ref{verification_theorem} with $b=y=0$ and $r=\rho$.

The following monotonicity result for the optimal threshold will be useful in the subsequent analysis. Its proof can be found in Appendix \ref{App_classical}. 
\begin{lemma} \label{lemma_useful_results}
The mapping $\rho \mapsto b^*_{\rho}$ is strictly decreasing on $(0,\infty)$. 
\end{lemma}

\subsection{Preliminary results for the value function}
\label{sec:bounds}

As a first step towards the construction of our candidate value function and control strategy, we present a useful stylised result for the control problem's value function $V(x;y)$ defined by \eqref{vf} in terms of its monotonicity with respect to the distress threshold $y$ and obtain robust bounds of $V(x;y)$ for all $y \geq 0$. 

\begin{lemma} \label{lem_mono}
For any $0 \leq y_1 \leq y_2 \leq \infty$, the value function $V(x;y)$ defined by \eqref{vf} satisfies 
$$
V_{r+q}(x) \equiv V(x;\infty) \leq V(x;y_2) \leq V(x;y_1) \leq V(x;0) \equiv V_{r}(x), 
\quad \text{for all } x \in \R.
$$
\end{lemma}

\begin{proof}
It follows by the definition \eqref{reward_classical}--\eqref{classical_value} of $V_{\rho}$ for a fixed discounting $\rho = r+q>0$, 
the definition \eqref{omega_clock} of the omega clock, which implies that $y \mapsto \omega^y_t$ is increasing and $rt \leq rt + \omega^y_t \leq (r + q)t$ for all $t \geq 0$ and $y \in \R$, 
and the definition \eqref{vf} of $V$, 
that
\begin{align*} 
V_{r+q}(x) 
&= \sup_{D \in \mathcal{A}} \E_x\bigg[ \int_0^{\tau^D_0} e^{-(r+q)t} dD_t \bigg] 
= \E_x\bigg[ \int_0^{\tau^D_0} e^{-(r+q)t} dD^{b^*_{r+q}}_t \bigg] \\
&\leq \E_x\bigg[ \int_0^{\tau^D_0} e^{-(rt + \omega^y_t)} dD^{b^*_{r+q}}_t \bigg]
\leq \sup_{D \in \mathcal{A}} \E_x\bigg[ \int_0^{\tau^D_0} e^{-(rt + \omega^y_t)} dD_t \bigg]
= V(x;y)
, \quad x \in \R .
\end{align*}
Similarly, we can prove the other inequalities.
\end{proof}

Even though we show in Lemma \ref{lem_mono} that $V(x;\cdot)$ is increasing, we unfortunately cannot immediately conclude the structure of the optimal dividend strategy. 
In fact, our analysis in the next section reveals that the exact form of the optimal control strategy differs significantly, both quantitatively and qualitatively, based on the value of the distress threshold $y$.

\section{Main Results} 
\label{sec:results}

A key structural feature of our model is that the solution to the optimal control problem depends critically on the distress threshold $y$, which partitions the surplus state space into a distress region $[0,y]$ and a no-distress region $(y,\infty)$. 
We identify three qualitatively distinct regimes based on the size of the distress region.

\subsection{Supercritical regime} \label{sec:supercritical}

In this regime, we consider large distress thresholds $y$, and we show that in this case, the firm behaves as if it is effectively always under distress. 
In particular, we show that there exists a separating distress threshold $y_u > 0$ such that all distress thresholds $y \in [y_u, \infty)$ are considered {\it high} and the optimal dividend strategy coincides with that of the fully penalised case (with discounting $r+q$ for all $t \geq 0$), and thus becomes independent of $y$. 
In such a case, the optimal strategy takes the form of a classical Skorokhod reflection at the constant boundary $b_{r+q}^*$ defined by \eqref{value_classical}. 
This implies that when the surplus process is in the waiting region $\mathcal{W}^y$, the decision maker does not exert control, while whenever the process is in the action region $\mathcal{D}^y$, the decision maker exerts the minimal amount of downward control to keep the process in the closure $\overline{\mathcal{W}^y}$ of the waiting region. 
In particular, we look for a separating distress threshold $y_u$ such that  
\begin{align}\label{regions_low_y}
\mathcal{W}^y = (0,b_{r+q}^*)
\quad \text{and} \quad 
\mathcal{D}^y = [b_{r+q}^*, \infty), 
\quad y \in [y_u, \infty).
\end{align}
Since we make the ansatz that the firm behaves like it is always in distress, it is suitable to make the ad hoc ansatz that $\mathcal{W}^y \subseteq (0,y_u)$, i.e.~$y_u\geq b_{r+q}^*$. 
Hence, we identify our candidate value function with $V_{r+q}(x)$ given by \eqref{value_classical} as a solution to the FBP \eqref{FBP_single_tresholda}--\eqref{FBP_single_tresholdc}. 

For the optimality of this candidate, we must ensure that $V_{r+q}$ and $b_{r+q}^*$ satisfy the conditions \eqref{mainthmcond1}--\eqref{mainthmcond2} of the Verification theorem.
In particular, by combining the FBP \eqref{FBP_single_tresholda}--\eqref{FBP_single_tresholdc} with the condition \eqref{mainthmcond2} for any $y \geq b_{r+q}^*$, we get the condition
\begin{align*}
\mathcal{L}V_{r+q}(x) - (r+q I_{\{x<y\}}) V_{r+q}(x) \leq 0, \quad x > 0,
\end{align*}
which is equivalent to
\begin{align} \label{yupper_def}
\begin{cases}
0\leq 0, & x \in (0,b^*_{r+q}], \\
\mu - (r+q) (x - b^*_{r+q} + \frac{\mu}{r+q})  \leq 0, & x \in (b^*_{r+q},y), \\
\mu - r (x - b^*_{r+q} + \frac{\mu}{r+q})  \leq 0, & x \in [y,\infty),
\end{cases} 
\quad \Leftrightarrow \quad 
y \geq y_u := b^*_{r+q} + \frac{\mu}r - \frac{\mu}{r+q},
\end{align}
where the equivalence follows from the latter inequality on the left-hand side.
It is thus clear from above that $y_u\geq b_{r+q}^*$, which satisfies our ansatz. 

Furthermore, we observe that $V_{r+q}$ violates condition \eqref{mainthmcond2} at $x=y$, when $y < y_u$ holds true, since 
\begin{equation} \label{Wy}
y < y_u
\quad \Leftrightarrow \quad 
\mathcal{L}V_{r+q}(y) - r V_{r+q}(y) = \mu - r (y - b^*_{r+q} + \frac{\mu}{r+q}) > 0, 
\end{equation}
for the uniquely defined threshold $y_u > 0$ in \eqref{yupper_def}. 

We are now ready to present our first main result 
(cf.~right panels in Figure \ref{fig:1}).

\begin{theorem}\label{theorem_optimal_yupper}
Let $y\geq y_u$, where $y_u$ is defined by \eqref{yupper_def}. 
Then we have that $V(x;y) = V_{r+q}(x)$  for all $x \geq 0$ and the optimal control is given by 
\begin{align*}
D^{b^*_{r+q}}_t=(x-b^*_{r+q})I_{\{x>b^*_{r+q}\}}+ L^{b^*_{r+q}}_t \big( X^{D^{b^*_{r+q}}} \big),
\end{align*}
where $V_{r+q}(\cdot)$ and $b_{r+q}^*$ are given by \eqref{value_classical}. 
\end{theorem}
\begin{proof}
Using our knowledge from the classical dividend problem with ruin at zero surplus in Section \ref{chapter_classical_case}, we have that $D^{b^*_{r+q}}\in \mathcal{A}$. 

Also, using our definition \eqref{yupper_def} of $y_u$ which implies that $y \geq y_u > b^*_{r+q}$, we conclude that the expected reward $J(x;y,D^{b^*_{r+q}})$ associated to $D^{b^*_{r+q}}$ solves the FBP \eqref{FBP_single_tresholda}--\eqref{FBP_single_tresholdc} with $\rho=r+q$, since the no-distress region $[y,\infty) \subseteq {D}^y = [b_{r+q}^*, \infty)$, thus does not interfere with the control strategy. 
This implies that $J(\cdot;y,D^{b^*_{r+q}}) = V_{r+q}(\cdot) \in \mathcal{C}^2(0,\infty)$.

Therefore, in order to apply the verification Theorem \ref{verification_theorem}, it remains to prove that $J(x;y,D^{b^*_{r+q}})$ satisfies conditions \eqref{mainthmcond1}-\eqref{mainthmcond2}.
The proof is split in the following three steps.

\vspace{1mm}
{\it Step 1. Condition \eqref{mainthmcond1} for $x\in (0,b_{r+q}^*]$.} 
Since $V_{r+q}''(b^*_{r+q})=0$, we can use Lemma \ref{Lemma_ODE_Sol}.(ii) to conclude that $V_{r+q}$ is concave on $(0,b^*_{r+q})$.
Hence, $V_{r+q}'(\cdot)$ is decreasing on $(0,b^*_{r+q})$ with $V_{r+q}'(b^*_{r+q})=1$, which follows from \eqref{FBP_single_tresholdb} and $V_{r+q}(\cdot) \in \mathcal{C}^2(0,\infty)$, which completes the proof.

\vspace{1mm}
{\it Step 2. Condition \eqref{mainthmcond1} for $x\in (b_{r+q}^*,\infty)$.} 
This is a straightforward consequence of the construction of $V_{r+q}$, which implies that $V_{r+q}'(x)=1$ for all $x\in (b_{r+q}^*,\infty)$.

\vspace{1mm}
{\it Step 3. Condition \eqref{mainthmcond2} for all $x \geq 0$.} 
This follows directly from the construction of $y_u$ and \eqref{yupper_def}. 
\end{proof}

In the following, we turn our attention to reflection strategies at a {\it distress level-dependent} boundary $b^*(y)$ -- still in the same class as before, but now the threshold depends on $y$ -- when $y < y_u$. 
The inequality in \eqref{Wy} motivates us to consider the case such that $y \in \mathcal{W}^y$, i.e. 
$b^*(y) > y$.

\subsection{Subcritical regime} \label{sec_subcritical}

In this regime, we consider small distress thresholds $y$, so that distress is only triggered at low surplus levels. 
We show that there exists a separating distress threshold $y_l > 0$ such that all distress thresholds $y \in (0, y_l]$ are considered {\it low} and the optimal dividend strategy takes the form of a classical Skorokhod reflection at a boundary function $b^*(y)$, which depends explicitly on $y$. 
This implies that when the surplus process is in the waiting region $\mathcal{W}^y$, the decision maker does not exert control, while whenever the process is in the action region $\mathcal{D}^y$, the decision maker exerts the minimal amount of downward control to keep the process in the closure $\overline{\mathcal{W}^y}$ of the waiting region. 
In particular, we look for a separating distress threshold $y_l$ such that   
$$
\mathcal{W}^y = (0,b^*(y))
\quad \text{and} \quad 
\mathcal{D}^y = [b^*(y), \infty) , 
\quad y \in (0, y_l].
$$

Given that for $y=0$, we have the classical dividend problem of Section \ref{chapter_classical_case} with optimal dividend boundary $b_{r}^*$ defined by \eqref{value_classical} for $\rho=r$, we define $b^*(0) := b_{r}^* > 0$. 
For small distress thresholds $y$, we thus expect to have $b^*(y)>y$. 
Hence, for any $y>0$, we construct our candidates $(w(\cdot;y),b^*(y))$ for the value of this strategy and the optimal dividend boundary function, respectively, under the ansatz that $b^*(y)>y$, by solving the following associated FBP 
\begin{align}
&\tfrac12 \sigma^2 w''(x;y)+\mu w'(x;y) - (r+qI_{\{x< y\}})w(x;y) = 0, \quad x\in (0,y)\cup(y,b^*(y)) \label{FBP_lowya}\\
&w(x;y) = w(b^*(y);y) + x - b^*(y), \qquad \qquad \qquad \qquad \quad x\in [b^*(y),\infty) \label{FBP_lowyb}\\ &w(0+;y)=0, \label{FBP_lowyc}\\ 
&w(\cdot;y) \in \mathcal{C}^2 \big((0,y) \cup (y,\infty) \big) \cap  \mathcal{C}^1(0,\infty) \label{FBP_lowyd}.
\end{align}
Notice that, at this stage of the analysis, we refrain from assuming the $\mathcal{C}^2$-regularity of the value $w(\cdot;y)$ at the distress threshold $y$ in \eqref{FBP_lowyd}, given that the state-dependent discount rate $x \mapsto r+qI_{\{x< y\}}$ is discontinuous at $\{y\}$ (cf.~\eqref{FBP_single_tresholdc} for the classical case).

Before commencing the analysis, we define the functions
\begin{align}
\label{def:delta}
\begin{split}
\delta(y) := &\psi_{r+q}(y) - \phi_{r+q}(y), \quad 
\eta(y;b) := \psi'_r(b)\phi_r(y) - \phi'_r(b) \psi_r(y), \quad \\ 
&\psi_\rho(y) := e^{\gamma_1(\rho) \,y}, \quad 
\phi_\rho(y) := e^{\gamma_2(\rho) \,y}, \quad 
y, b, \rho > 0 . 
\end{split}
\end{align}
for $\gamma_i(\cdot),i=1,2$, defined by \eqref{gamma_def}.
We now present the solution to the FBP \eqref{FBP_lowya}--\eqref{FBP_lowyd} and properties of the solution; the proof can be found in Appendix \ref{App_sec_subcritical}.

\begin{lemma} \label{sol:FBP1}
For any $y>0$, the free-boundary problem \eqref{FBP_lowya}--\eqref{FBP_lowyd} admits a unique solution given by 
\begin{align} \label{w_sol_simple_def}
w(x;y)= 
\begin{cases}
K_1(y) \, e^{\gamma_1(r+q) \, x} + K_2(y) \, e^{\gamma_2(r+q) \, x}, \quad &x \in [0,y),\\
K_3(y) \, e^{\gamma_1(r) \, x} + K_4(y) \, e^{\gamma_2(r) \, x}, \quad &x \in [y,b^*(y)),\\
w(b^*(y);y) + x - b^*(y), \quad &x\in [b^*(y),\infty),
\end{cases}
\end{align}
with $\gamma_1(\cdot)$, $\gamma_2(\cdot)$ defined by \eqref{gamma_def}, $K_1(y), \ldots, K_4(y)$ defined by 
\begin{align}\label{constants_simple}
\begin{split}
K_1(y) = & - K_2(y) =  
\frac{\psi_r(y) \eta' (y; b^*(y)) - \psi'_r(y)\eta (y; b^*(y))}{\psi'_r(b^*(y))\bigl(\delta(y) \eta'(y; b^*(y)) - \delta'(y) \eta(y; b^*(y))\bigr)}, \\
K_3(y) = 
&\frac{\phi'_r(b^*(y)) \bigl( \delta \psi'_r - \delta' \psi_r \bigr)(y) + \delta(y) \eta'(y; b^*(y)) - \delta'(y) \eta(y; b^*(y))}{\psi'_r(b^*(y))\bigl(\delta(y) \eta'(y; b^*(y)) - \delta'(y) \eta(y; b^*(y))\bigr)}, \\ 
K_4(y) = 
&\frac{ \big(\psi_r \delta' -  \psi'_r \delta \big)(y)}{\delta(y) \eta'(y; b^*(y)) - \delta'(y) \eta(y; b^*(y))},
\end{split}
\end{align}
where $\delta$, $\eta$, $\psi_\rho$ and $\phi_\rho$ are defined by \eqref{def:delta}, and 
\begin{align} \label{def_deltay}
\begin{split}
b^*(y) &=  y + \Delta(y), \quad \text{with} \quad  
\Delta(y) := \frac{1}{\gamma_1(r)-\gamma_2(r)}\log\left(\frac{(\delta'(y)-\gamma_1(r)\delta(y)) \, \gamma_2^2(r)}{(\delta'(y)-\gamma_2(r)\delta(y)) \, \gamma_1^2(r)}\right),
\end{split}
\end{align}
such that $\Delta(y) > 0$ for all $0 < y \leq y_u$ and $y \mapsto b^*(y)$ is strictly increasing on $(0,\infty)$.
\end{lemma} 

Even though the solution $w(x;y)$ to the FBP \eqref{FBP_lowya}--\eqref{FBP_lowyd} is obtained in Lemma \ref{sol:FBP1} for all $y>0$, it cannot serve as our candidate value function for all $y>0$. 
According to the following result, which is proved in Appendix \ref{App_sec_subcritical}, the function $w(x;y)$ given by Lemma \ref{sol:FBP1} satisfies condition \eqref{mainthmcond1} of the verification Theorem \ref{verification_theorem} (and can act as our candidate value function) if and only if $y \in (0,y_l]$. 
An additional important result is the complete and unique characterisation of the latter separating distress threshold $y_l$, which thus defines the subcritical regime of {\it low distress levels} $y \in (0,y_l]$.

\begin{lemma}\label{lemma_yl}
Let $(w(\cdot;y),b^*(y))$ be given by \eqref{w_sol_simple_def}--\eqref{def_deltay} and recall $\delta(y)$ from \eqref{def:delta}. 
Then, there exists a unique solution $y_l \in (0,y_u)$ to the equation $f(y) = 0$, where $f:(0,y_u) \to \R$ is defined by
\begin{align} \label{ylower_eq_charac} 
f(y) := \left(- \frac{\gamma_1(r)}{\gamma_2(r)}\delta'(y) + \gamma_1(r)\delta(y) \right) \left(\frac{\gamma_2^2(r)}{\gamma_1^2(r)} \, \frac{\delta'(y) - \gamma_1(r) \delta(y)}{\delta'(y) - \gamma_2(r) \delta(y)} \right)^{\frac{\gamma_1(r)}{\gamma_1(r) - \gamma_2(r)}} - \delta'(b^*_{r+q}) .
\end{align}
and $y_u$ is defined in \eqref{yupper_def}. 
Furthermore, we have that $y_l\in (b^*_{r+q},y_u)$ and 
\begin{align} \label{ylower_def}
\begin{split}
\begin{cases}
w'(b^*_{r+q};y)\geq 1,\quad y\in (0,y_l),\\
w'(b^*_{r+q};y)< 1,\quad y\in (y_l,\infty).
\end{cases}
\end{split}
\end{align}
\end{lemma}

\begin{corollary} \label{help_corr}
Recall the unique solution $(w(\cdot;y),b^*(y))$ to the free-boundary problem \eqref{FBP_lowya}-\eqref{FBP_lowyd} given by Lemma \ref{sol:FBP1}, and the unique solution $y_l$ of \eqref{ylower_def} in Lemma \ref{lemma_yl}. 
Then, we have 
\begin{align*}
K_1(y_l) = - K_2(y_l) = \frac{1}{\gamma_1(r+q) e^{\gamma_1(r+q) b^*_{r+q}} - \gamma_2(r+q) e^{\gamma_2(r+q)b^*_{r+q}}} 
\quad \text{and} \quad 
w(x;y_l)
\;  x \in [0,b^*_{r+q}],
\end{align*}
where $V_{r+q}$ is the solution \eqref{value_classical} of the classical dividend problem \eqref{classical_value} with $\rho = r+q$.
\end{corollary}

\begin{proof}
Using the continuity of $y \mapsto w'(b^*_{r+q}; y)$ at $y=y_l$ (recall the continuity of $K_1(\cdot), \dots, K_4(\cdot)$ on $(0, \infty)$), we see from \eqref{ylower_def} that 
$w'(b^*_{r+q}; y_l)=1$ and from \eqref{value_classical} that $V_{r+q}'(b^*_{r+q})=1$ as well. 
Hence, using \eqref{w_sol_simple_def}--\eqref{constants_simple} together with \eqref{def:delta}, we observe from \eqref{value_classical} that
\begin{align*}
w'(b^*_{r+q}; y_l) 
&= K_1(y_l) \left(\psi'_{r+q}(b^*_{r+q}) - \phi'_{r+q}(b^*_{r+q}) \right) \\
&= 1 
= \frac{\left(\psi'_{r+q}(b^*_{r+q})-\phi'_{r+q}(b^*_{r+q})\right)}{\gamma_1(r+q) e^{\gamma_1(r+q) b^*_{r+q}} - \gamma_2(r+q) e^{\gamma_2(r+q)b^*_{r+q}}} 
= V_{r+q}'(b^*_{r+q}),
\end{align*}
which implies the desired expressions for $K_1(y_l)$ and $K_2(y_l)$ and consequently we get that $w(x; y_l) = V_{r+q}(x)$, for all $x \in [0,b^*_{r+q}]$.
\end{proof}

We are now in position to prove that the solution $w(\cdot;y)$ to the FBP \eqref{FBP_lowya}--\eqref{FBP_lowyd} obtained in Lemma \ref{sol:FBP1} does indeed identify with the value function of the control problem when we are in the subcritical regime of low distress levels $y \in (0,y_l]$. 
We also explicitly construct the optimal control strategy corresponding to the waiting and action regions defined in \eqref{regions_low_y}
(cf.~left panels in Figure \ref{fig:1}).

\begin{theorem} \label{thm:sub}
Recall the unique solution $(w(\cdot;y), b^*(y))$ of the free-boundary problem \eqref{FBP_lowya}--\eqref{FBP_lowyd} obtained in Lemma \ref{sol:FBP1}, and the unique solution $y_l$ of the equation \eqref{ylower_def} in Lemma \ref{lemma_yl}. 
Then, for any $y\in(0,y_l]$, we have that the value function of the singular control problem \eqref{vf} is given by $V(x;y) = w(x;y)$ for all $x\geq 0$ and the control process $D^{b^*(y)}$ defined by 
\begin{align*}
D^{b^*(y)}_t := (x-b^*(y)) I_{\{x>b^*(y)\}} + L^{b^*(y)}_t(X^{D^{b^*(y)}}), 
\quad t \geq 0,
\end{align*}
is admissible and optimal.
\end{theorem}

\begin{proof}
We first note that $D^{b^*(y)}$ is admissible, i.e.~$D^{b^*(y)}\in \mathcal{A}$; see Section \ref{scorohod_sec} for details. 

Then, we can show that the expected reward $J(x;y,D^{b^*(y)})$ associated to $D^{b^*(y)}$ equals the solution $w(x;y)$ to the FBP \eqref{FBP_lowya}--\eqref{FBP_lowyd}; the proof of this part is omitted since it is a simpler version of the arguments included in the proof of Theorem \ref{thm:sub} (for the more complicated critical regime).
Thanks to this, we therefore have that $J(\cdot;y,D^{b^*(y)}) \in \mathcal{C}^2((0,y) \cup (y,\infty)) \cap \mathcal{C}^1(0,\infty)$.

Therefore, in order to apply the verification Theorem \ref{verification_theorem}, it remains to prove that $J(\cdot;y,D^{b^*(y)}) = w(\cdot;y)$ satisfies conditions \eqref{mainthmcond1}-\eqref{mainthmcond2}, which take the form
\begin{align}
w'(x;y) &\geq 1, \quad x\in (0,\infty), \label{mainthmcond1sub}  \tag{I'} \\
g(x;y) := \mathcal{L} w(x;y) \, I_{\{x\not=y\}} -\big(r + q \, I_{\{x<y\}} \big) \, w(x;y) &\leq 0, \quad x\in (0,\infty). \label{mainthmcond2sub}  \tag{II'} 
\end{align}
The proof is split in the following four steps.

\vspace{1mm}
{\it Step 1. Condition \eqref{mainthmcond1sub} for $x\in (0,b^*(y))$.} 
We begin by observing from Lemma \ref{lemma:yl}.(vi), \eqref{w_sol_simple_def}--\eqref{constants_simple} and \eqref{gamma_def}, that 
$$
w'(0+;y) = K_1(y) \big( \gamma_1(r+q) - \gamma_2(r+q) \big) > 0.
$$
Thus, we can use Lemma \ref{Lemma_ODE_Sol}.(i) for $\rho=r+q$, $\ubar{x}=z_0=0$ and $z_1 = w'(0+;y)>0$, to see that 
$$
w(x;y) > 0 \quad \text{and} \quad w'(x;y) > 0 , \quad \text{for all } x \in (0,y].
$$ 
Then, given that $y < b^*(y)$, thanks to Lemma \ref{sol:FBP1}, and that $w''(b^*(y)-;y)=0$, we can also use Lemma \ref{Lemma_ODE_Sol}.(ii) for $\rho=r$, $\ubar{x}=y>0$ and $x^*=b^*(y)$, to conclude that 
\begin{equation} \label{w:concave>y}
w(\cdot;y) \text{ is concave on $(y,b^*(y))$ and convex on $(b^*(y),\infty)$}.
\end{equation}

To complete the proof of \eqref{mainthmcond1sub}, we then examine the following cases: 

{\it Case 1: $y \leq b^*_{r+q}$}.
Using the expression of $b^*_{r+q}$ in \eqref{value_classical}, we can directly verify that $w(x;y)=K_1(y)(e^{\gamma_1 x}-e^{\gamma_2 x})$ is concave on $(0,y)$, which combined with \eqref{w:concave>y} yields that $w(\cdot;y)$ is concave on $(0,y)\cup (y,b^*(y))$. Combining this with $w'(b^*(y);y)=1$ gives $w'(x;y) \geq 1$ on $(0,b^*(y))$.

{\it Case 2: $b^*_{r+q} < y$}.
Using again the expression of $b^*_{r+q}$ in \eqref{value_classical} and the property in \eqref{w:concave>y}, we can directly verify that $w(\cdot;y)$ is concave on $(0,b^*_{r+q}) \cup (y,b^*(y))$ and convex on $(b^*_{r+q},y)$. 
Combining this with the fact that $y \leq y_l$ and Lemma \ref{lemma_yl}, which imply that $w'(b^*_{r+q};y) \geq 1 = w(b^*(y);y)$, we see that $w'(x;y)\geq 1$ on $(0,b^*(y))$. 

\vspace{1mm}
{\it Step 2. Condition \eqref{mainthmcond1sub} for $x\in [b^*(y),\infty)$.} 
This is a straightforward consequence of the construction of $w(\cdot;y)$, which implies that $w'(x;y)=1$ for all $x \in [b^*(y),\infty)$.

\vspace{1mm}
{\it Step 3. Condition \eqref{mainthmcond2sub} for $x\in (0, b^*(y))$.} 
This is a straightforward consequence of the construction of $w(\cdot;y)$, which implies that $g(x;y)=0$ for all $x \in (0, b^*(y))$.

\vspace{1mm}
{\it Step 4. Condition \eqref{mainthmcond2sub} for $x\in [b^*(y),\infty)$.} 
Using the fact that $y < b^*(y)$ and \eqref{FBP_lowyb} we have  
\begin{align*}
g(x;y) = \mu - r 
w(x;y) 
\quad \text{and} \quad
g'(x;y) = - r 
< 0, 
\quad x\in (b^*(y),\infty)
\end{align*}
Combining this with $g(b^*(y);y)=0$, which holds true thanks to $w(x;y) \in \mathcal{C}^2$ at $x=b^*(y)$ and \eqref{FBP_lowya}, 
yields that $g(x;y) \leq 0$ for all $x\in [b^*(y),\infty)$.
\end{proof}

\begin{remark}
Recall that in the beginning of this section, we make the ansatz that $b^*(y)>y$, based on which we formulate the FBP \eqref{FBP_lowya}--\eqref{FBP_lowyd} that gives the value function of the control problem (cf.~Theorem~\ref{thm:sub}). 
Given the results in Lemmata \ref{sol:FBP1} and \ref{lemma_yl}, this is always satisfied in the subcritical regime of low distress levels $y \in (0,y_l]$, since $y_l < y_u$. This makes the above analysis complete.
\end{remark}

Recall from Lemma \ref{lemma_useful_results} that as the ``traditional discount rate" $\rho$ increases in the classical dividend problem with ruin at zero, the optimal dividend barrier decreases so the action region $[b^*_{\rho}, \infty)$ expands. 
This implies that the decision maker  becomes more impatient, must optimally become more proactive, and pay dividends sooner, i.e.~when the surplus reaches a lower threshold. 

Surprisingly, increasing the distress threshold $y$ leads to a statistically higher effective discount rate $r + q I_{\{X^D<y\}}$ in our dual-ruin dividend problem \eqref{vf}, yet Lemma \ref{sol:FBP1} shows that the action region $[b^*(y), \infty)$ shrinks. 
This implies that the decision maker becomes less proactive and should optimally postpone dividend payments, i.e.~paying dividends only when the surplus reaches a higher threshold.
This behaviour can be explained by the fact that a higher dividend barrier $b^*(y)$ allows the surplus process to spend more time in the desirable no-distress region $(y,b^*(y)]$, where discounting is lower. In this way, delaying dividend payments mitigates the expected impact of the higher distress-induced discounting, reversing the usual monotonicity observed in classical dividend problems with constant discounting.

In summary, the above results imply that, contrary to the classical dividend problem with constant discounting, a statistically increasing effective discount rate in our model does not necessarily translate into greater impatience of the decision maker, as measured by earlier dividend payments.
Instead, the optimal policy becomes more conservative, with dividends being postponed to higher surplus levels.
This constitutes a fundamental departure from the classical theory and highlights the qualitatively different role played by our additional occupation-time-based ruin mechanism in singular control problems, compared to simply facing the traditional ruin at zero.
It is also worth noting that such a phenomenon does not arise in optimal stopping problems with a similar occupation-time-based ruin mechanism; see, for instance, \cite{neo_omega2018}.

Finally, observe from \eqref{ylower_def} that even though the action region for the borderline case of distress level $y=y_l$ is a connected interval, i.e. 
$$
\mathcal{D}^{y_l} = [b^*(y_l), \infty), \quad \text{(thanks to the analysis of this section)}
$$
there exists a point $b^*_{r+q} \in \mathcal{W}^{y_l} = (0, b^*(y_l))$, i.e.~in the waiting region, such that 
$$
V'(b^*_{r+q};y) 
= J'(b^*_{r+q};y,D^{b^*(y_l)}) 
= w'(b^*_{r+q};y_l) 
= 1.
$$
This motivates us to conjecture that for $y>y_l$, there could be an additional action region around the point $b^*_{r+q}$.
We prove that this is indeed the case in the following section.

\subsection{Critical regime} \label{sec_critical_regime}

In this regime, we consider intermediate distress thresholds $y$, and prove that the optimal strategy exhibits a genuinely new structure, with disconnected action and inaction regions. 
This is the most interesting and novel behaviour, emerging due to the interplay among the occupation time penalisation, ruin at zero, and surplus dynamics. This regime marks a qualitative departure from the usual known results in singular control and risk theory.
For $y \in (y_l, y_u)$, we conjecture that the action region is disconnected and takes the form 
\begin{align*}
\mathcal{D}^{y} = [a(y),\underline{b}(y)] \cup [\overline{b}(y), \infty),
\end{align*}
for some $a(y), \underline{b}(y)$ and $\overline{b}(y)$ to be found. 

Using our observation at the end of Section \ref{sec_subcritical} (Subcritical regime) that  
$$
w'(x;y_l) = 1 
\quad \text{for} \quad 
x \in \{b^*_{r+q}\} \cup [b^*(y_l), \infty), 
\quad \text{such that} \quad 
b^*_{r+q} < y_l < b^*(y_l), 
$$ 
we conjecture that the additional component $[a(y), \underline{b}(y)]$ of the action region contains the point $\{b^*_{r+q}\}$ and that for (at least sufficiently small) $y \in (y_l,y_u)$ we have  $\underline{b}(y) < y < \overline{b}(y)$.

Observing also that, if the process $X_{t_0} \in [0, \underline{b}(y)]$ for some $t_0 \geq 0$, then it will never exit this interval again, i.e.~$0 \leq X_t \leq \underline{b}(y) < y$, for all $t \geq t_0$. 
This implies that the decision maker will discount with the constant rate $r+q$ for all $t \geq t_0$. 
In such a case, the optimal strategy was obtained in Theorem \ref{theorem_optimal_yupper} (cf.~Section \ref{sec:supercritical}). 
Hence, the part $[0, \underline{b}(y)]$ of the state space should partition into the action region $[b^*_{r+q}, \underline{b}(y)]$ and the waiting region $[0, b^*_{r+q})$, and the value function should be therefore given by  
$$
V(X_t;y) = J(X_t;y,D^{b^*_{r+q}})=V_{r+q}(X_t), \quad \text{for all } t \geq t_0.
$$ 
We can thus conclude that the conjectured action region should satisfy $a(y) = b^*_{r+q}$, which fixes the lower threshold of our candidate for the optimal strategy. 

Hence, for any $y \in (y_l, y_u)$, we construct our candidate as the triplet $(w(\cdot;y), \underline{b}^*(y), \overline{b}^*(y))$ for the value of this strategy and the two free optimal dividend boundary functions, respectively, under the ansatz that $b^*_{r+q} < \underline{b}^*(y) < y < \overline{b}^*(y)$ and 
\begin{equation} \label{ac_wait_medy}
\mathcal{W}^y = \big(0,b^*_{r+q} \big) \cup \big(\underline{b}^*(y), \overline{b}^*(y) \big) 
\quad \text{and} \quad 
\mathcal{D}^y = \big[b^*_{r+q},  \underline{b}^*(y) \big] \cup \big[\overline{b}^*(y), \infty \big) , 
\quad y \in (y_l, y_u),
\end{equation}
by solving the following associated FBP 
\begin{align}
& w(x;y) = V_{r+q}(x), \qquad \qquad \qquad \qquad \qquad \qquad \qquad \qquad x\in (0,\ubar{b}^*(y)]\label{FBP_medya}\\
& \tfrac{\sigma^2}{2} w''(x;y) + \mu w'(x;y) - (r + q I_{\{x< y\}}) w(x;y) = 0, \qquad \, x \in (\ubar{b}^*(y),y) \cup (y,\bar{b}^*(y)) \label{FBP_medyc}\\
& w(x;y) = w(\bar{b}^*(y);y) + x - \bar{b}^*(y), \qquad \qquad \qquad \qquad \quad\;\; x\in [\bar{b}^*(y),\infty) \label{FBP_medyd}\\
& w \in \mathcal{C}^2((0,\ubar{b}^*(y)) \cup (\ubar{b}^*(y),y) \cup (y,\infty)) \cap \mathcal{C}^1(0,\infty) \label{FBP_medye}.
\end{align}
Notice that, at this stage of the analysis, we refrain from assuming the $\mathcal{C}^2$-regularity of the value $w(\cdot;y)$ at $\{y,\ubar{b}^*(y)\}$. 
This is due to the discontinuity of the state-dependent discount rate $x \mapsto r + q I_{\{x< y\}}$ at the distress threshold $y$ (as in Section \ref{sec_subcritical} for the  subcritical regime), 
and the fact that our admissible controls can only decrease the state process \eqref{control_SDE}, while $\ubar{b}^*(y)$ is a lower boundary of the waiting region (i.e.~cannot be a reflecting one). 
Contrary the $C^2$-regularity of $w(\cdot;y)$ at $\{b^*_{r+q}\}$ results directly from the $C^2$-regularity of $V_{r+q}(\cdot)$ from \eqref{value_classical}. Before commencing the analysis, we define the functions

\begin{align} \label{constants_complex}
\begin{split}
e_1(b) 
&:=\frac{(\phi_{r+q}-\phi'_{r+q}V_{r+q})(b)}{(\psi'_{r+q}\phi_{r+q}-\psi_{r+q}\phi'_{r+q})(b)}, 
\quad 
e_2(b) 
:=\frac{(\psi'_{r+q}V_{r+q}-\psi_{r+q})(b)}{(\psi'_{r+q}\phi_{r+q}-\psi_{r+q}\phi'_{r+q})(b)}, \\ 
e_3(b,y) 
&:= \frac{e_1(b)(\phi_r\psi'_{r+q} - \phi'_r\psi_{r+q})(y) + e_2(b)(\phi_r\phi'_{r+q} - \phi'_r\phi_{r+q})(y)}{(\psi'_{r}\phi_{r}-\psi_{r}\phi'_{r})(y)}, \\ 
e_4(b,y) 
&:= \frac{e_1(b)(\psi'_r\psi_{r+q}-\psi_r\psi'_{r+q})(y) + e_2(b)(\psi'_r\phi_{r+q}-\psi_r\phi'_{r+q})(y)}{(\psi'_{r}\phi_{r}-\psi_{r}\phi'_{r})(y)}
, \quad y,b > 0,
\end{split}
\end{align}
for $\psi_\rho$ and $\phi_\rho$ defined by \eqref{def:delta}, which are well-defined since the denominators are all strictly positive.
We now prove in the following result that the FBP \eqref{FBP_medya}--\eqref{FBP_medye} admits a unique solution for all $y\in(y_l,y_u)$. Its proof can be found in Appendix \ref{App_critical}.

\begin{lemma} \label{fbp2_ex}
Recall $V_{r+q}$ and $b^*_{r+q}$ from \eqref{value_classical} for $\rho=r+q$, 
$y_u$ defined in \eqref{yupper_def}, 
the unique solution $y_l$ of \eqref{ylower_eq_charac}, 
and $e_1, \ldots, e_4$ defined by \eqref{constants_complex}. 
For any $y\in(y_l,y_u)$, the free-boundary problem \eqref{FBP_medya}-\eqref{FBP_medye} admits a unique solution given by 
\begin{align} \label{w_sol_double_def}
w(x;y)= 
\begin{cases}
V_{r+q}(x)\quad &x \in [0,\ubar{b}^*(y)),\\
E_1(y) \, e^{\gamma_1(r+q) \, x} + E_2(y) \, e^{\gamma_2(r+q) \, x}, \quad &x \in [\ubar{b}^*(y),y),\\
E_3(y) \, e^{\gamma_1(r) \, x} + E_4(y) \, e^{\gamma_2(r) \, x}, \quad &x \in [y,\bar{b}^*(y)),\\
w(\bar{b}^*(y);y) + x - \bar{b}^*(y), \quad &x\in [\bar{b}^*(y),\infty),
\end{cases}
\end{align} 
with $\gamma_1(\cdot)$, $\gamma_2(\cdot)$ defined by \eqref{gamma_def}, $E_1(y), \ldots, E_4(y)$ are given by 
$E_1(y) = e_1(\ubar{b}^*(y))$, $E_2(y) 
= e_2(\ubar{b}^*(y))$, $E_3(y) = e_3(\ubar{b}^*(y),y)$ and $E_4(y) = e_4(\ubar{b}^*(y),y)$,  
$\ubar{b}^*(y)$ is the unique solution to $H(\ubar{b}^*(y),y) = 1$ on $(b^*_{r+q},y)$, with 
\begin{align} \label{H_eq_midy}
H(b,y):= 
(\gamma_1(r)-\gamma_2(r)) \Big(-\frac{\gamma_2(r)}{\gamma_1(r)}\Big)^{\frac{\gamma_1(r)+\gamma_2(r)}{\gamma_1(r)-\gamma_2(r)}} e_3(b,y)^{\frac{-\gamma_2(r)}{\gamma_1(r)-\gamma_2(r)}}(-e_4(b,y))^{\frac{\gamma_1(r)}{\gamma_1(r)-\gamma_2(r)}},
\end{align}
and 
\begin{align} \label{upperb_solv}
\bar{b}^*(y) 
= \frac{1}{\gamma_1(r)-\gamma_2(r)} \log\left(-\frac{\gamma_2^2(r)e_4(\ubar{b}^*(y),y)}{\gamma_1^2(r)e_3(\ubar{b}^*(y),y)}\right) \, > y .
\end{align}
\end{lemma} 

It is worth noting that there is no explicit dependence of the functions $E_1(y)$ and $E_2(y)$ on the distress level $y$ -- their dependence on $y$ is only via the boundary $\ubar{b}^*(y)$.
Moreover, note that the upper boundary function $\bar{b}^*(y)$ is given explicitly in \eqref{upperb_solv}, whereas the lower one $\ubar{b}^*(y)$ is completely characterised as the unique solution to the equation $H(\ubar{b}^*(y),y) = 1$ (cf.~\eqref{H_eq_midy}).
All these properties result from our tailored proof technique (see Appendix \ref{App_critical}) for finding the unique solution to the system of 6 equations and 6 unknowns implied by the FBP \eqref{FBP_medya}--\eqref{FBP_medye}. 
In particular, our proof involves: 
(a) solving an auxiliary boundary value problem (with 4 equations and 4 unknowns $E_1, \ldots, E_4$) for an arbitrary and fixed pair $(\ubar{b},y)$ satisfying $\ubar{b} < y$, to compute the functions $E_1, \ldots, E_4$ in terms of $e_i$ in \eqref{constants_complex} and the arbitrary pair $(\ubar{b},y)$; 
(b) choosing $\ubar{b} = \underline{b}^*(y)$ and finding $\bar{b}^*(y)$ simultaneously, as the unique solution to a highly non-linear two-dimensional system of two equations;
(c) showing that this system is separable, i.e.~we can express $\bar{b}^*(y)$ as a function of $\underline{b}^*(y)$, and finally solve a single equation for $\underline{b}^*(y)$.

Given the existence of the solution $w(\cdot;y)$ to the FBP \eqref{FBP_medya}--\eqref{FBP_medye} in Lemma \ref{fbp2_ex}, we are now in position to prove that $w(\cdot;y)$ identifies with the control problem's value function when we are in the critical regime of intermediate distress levels $y \in (y_l,y_u)$. 
We also explicitly construct the optimal control strategy corresponding to the waiting and action region defined in \eqref{ac_wait_medy}
(cf.~middle panels in Figure \ref{fig:1}).

\begin{theorem} \label{thm:sub}
Recall the unique solution $(w(\cdot;y), \ubar{b}^*(y), \bar{b}^*(y))$ of the free-boundary problem \eqref{FBP_medya}--\eqref{FBP_medye} obtained in Lemma \ref{fbp2_ex}, the unique solution $y_l$ of the equation \eqref{ylower_def} in Lemma \ref{lemma_yl} and the value $y_u$ defined by \eqref{yupper_def}. 
Then, for any $y\in(y_l,y_u)$, we have that the value function of the singular control problem \eqref{vf} is given by $V(x;y) = w(x;y)$ for all $x\geq 0$ and the control process $D^{b_{r+q}^*,\ubar{b}^*(y),\bar{b}^*(y)}$ defined by 
\begin{align*}
D^{b_{r+q}^*,\ubar{b}^*(y),\bar{b}^*(y)}_t 
&:= (x-\bar{b}^*(y)) I_{\{x>\bar{b}^*(y)\}} 
+ (x-b^*_{r+q}) I_{\{x \in (b^*_{r+q}, \ubar{b}^*(y))\}} + (\ubar{b}^*(y) - b_{r+q}^*) I_{\{t\geq \zeta_{\ubar{b}^*(y)}\}} \\ 
&\quad + L^{b_{r+q}^*}_t\big(X^{D^{b^*_{r+q},\ubar{b}^*(y),\bar{b}^*(y)}}\big) 
+ L^{\bar{b}^*(y)}_t\big(X^{D^{
b^*_{r+q},\ubar{b}^*(y),\bar{b}^*(y)}}\big),
\quad t \geq 0,
\end{align*}
is admissible and optimal, where we also define
$$
\zeta_{\ubar{b}^*(y)}
:= 
\big\{ t \geq 0 \,:\, X^{D^{b_{r+q}^*,\ubar{b}^*(y),\bar{b}^*(y)}}_t = \ubar{b}^*(y) \big\}.
$$
\end{theorem}

\begin{proof}
We first note that $D^{b_{r+q}^*,\ubar{b}^*(y),\bar{b}^*(y)}$ is admissible, i.e.~$D^{b_{r+q}^*,\ubar{b}^*(y),\bar{b}^*(y)} \in \mathcal{A}$; see Appendix \ref{scorohod_sec_jump} for its construction and further details. 
The remainder of the proof is split into the following two main steps. 

\vspace{1mm}
{\it Step 1. Proof that the expected reward $J(x;y,D^{b_{r+q}^*,\ubar{b}^*(y),\bar{b}^*(y)})$ associated to $D^{b_{r+q}^*,\ubar{b}^*(y),\bar{b}^*(y)}$ equals the solution $w(x;y)$ to the FBP \eqref{FBP_medya}--\eqref{FBP_medye}}.
For ease of notation, we denote by $\bar{D} = D^{b_{r+q}^*\ubar{b}^*(y),\bar{b}^*(y)}$ and we prove the required result by considering separately the following cases:

\vspace{1mm}
{\it Step 1(a). Suppose that $x=0$}.
In this case, it is trivially seen that $w(0;y)=0=J(0;\bar{D})$ by their definitions. 

\vspace{1mm}
{\it Step 1(b). Suppose that $x\in (0,b^*_{r+q}]$}. 
We firstly see from \eqref{FBP_medya} and \eqref{FBP_medye} that $w(x;y)=V_{r+q}(x)$ and $w(\cdot;y) \in  \mathcal{C}^2(0,\ubar{b}^*(y))$. 
An application of It\^{o}'s formula for semi-martingales \cite[p.~74]{Peskir} thus gives 
\begin{align*}
e^{- (r+q) (t \wedge \tau_0^{\bar{D}})} w(X^{\bar{D}}_{t\wedge\tau_0^{\bar{D}}};y) - w(x;y)
&= \int_0^{t\wedge \tau_0^{\bar{D}}} e^{-(r+q)s} (\mathcal{L}-(r+q)) w(X^{\bar D}_s;y)ds \\
&\quad + \int_0^{t\wedge \tau_0^{\bar D}} e^{-(r+q)s} w'(X^{\bar D}_s;y) (dW_s - d\bar D_s).
\end{align*}
Recall from its construction (cf.~Appendix \ref{scorohod_sec_jump}) that the process $X^{\bar D}$ started from $x\in (0,b^*_{r+q}]$ is reflected downwards at $b^*_{r+q}$, such that $X^{\bar D}_t \leq b^*_{r+q}$ for all $t\geq 0$, hence $\bar D_t = L^{b^*_{r+q}}_t(X^{\bar D})$ for all $t\geq 0$, $\mathbb{P}$-a.s.. 
Combining this with the ODE \eqref{FBP_single_tresholda}, which implies that the first integral is equal to zero, we obtain by taking expectations that 
\begin{align}
w(x;y) 
&= \E_x \bigg[\int_0^{t \wedge \tau^{\bar D}_0} e^{-(r+q)s} w'(X^{\bar D}_s;y) dL^{b^*_{r+q}}_s(X^{\bar D}) \bigg] 
+ \E_x\left[ e^{-(r+q) (t\wedge\tau_0^{\bar D})} w(X^{\bar D}_{t\wedge\tau_0^D};y) \right] \notag\\ 
&= \E_x \bigg[ \int_0^{t\wedge\tau^{\bar D}_0} e^{-(r+q)s} dL^{b^*_{r+q}}_s(X^{\bar D}) \bigg] 
+ \E_x\left[ e^{-(r+q) (t\wedge\tau_0^{\bar D})} w(X^{\bar D}_{t\wedge\tau_0^D};y) \right], 
\label{wLb}
\end{align}
where we used that $I_{\{X^{\bar D}_t \neq b^*_{r+q}\}} dL^{b^*_{r+q}}_t(X^{\bar D}) \equiv 0$ and $w'(b^*_{r+q}) = 1$. 
Then, we observe that 
$$
\lim_{t\to\infty} \E_x \left[e^{-(r+q)(t\wedge\tau_0^{\bar D})} w(X^{\bar D}_{t\wedge\tau_0^{\bar D}};y) \right] = 0,
$$ 
thanks to the dominated convergence theorem since we have $X^{\bar D}_t\leq b^*_{r+q}$ for all $t\geq 0$ and $w(X^{\bar D}_{\tau_0^{\bar D}})I_{\{\tau^{\bar D}_0 < \infty)\}}=0$.
Thus, taking the limits as $t\to \infty$ in \eqref{wLb}, we get by the monotone convergence theorem that
\begin{align*}
w(x;y) = 
\E_x \bigg[ \int_0^{\tau^{\bar D}_0} e^{-(r+q)s} dL^{b^*_{r+q}}_s(X^{\bar D}) \bigg] 
=J(x;\bar D).
\end{align*}

\vspace{1mm}
{\it Step 1(c). Suppose that $x\in (b_{r+q}^*,\ubar{b}^*(y)]$}. 
In this case, the definition of $\bar D$ and Step 1(b) imply that
\begin{align*}
J(x;\bar D) 
= x -b_{r+q}^* + J(b_{r+q}^*;\bar D) 
= x -b_{r+q}^* + w(b_{r+q}^*;y) 
= w(x;y).
\end{align*}

\vspace{1mm}
{\it Step 1(d). Suppose that $x\in (\ubar{b}^*(y),\bar{b}^*(y)]$}. 
Repeating the same arguments as in Step 1(b) with $\tau_{\ubar{b}^*(y)}^{\bar D}$ instead of $\tau_0^{\bar D}$ and the discounting $\lambda(t):=\int_0^{t} (r + qI_{\{X^D_s<y\}}) ds$, $t \geq 0$, we get 
\begin{align}
w(x)=\E_x\left(\int_0^{\tau_{\ubar{b}^*(y)}^{\bar D}} e^{-\lambda(t)}dL_t^{\bar{b}^*(y)}+ e^{-\lambda(\tau_{\ubar{b}^*(y)}^{\bar D})}w(X_{\tau_{\ubar{b}^*(y)}^{\bar D}})\right).
\end{align}
Using the previous cases as well as the Markov property, we get that
\begin{align*}
    w(x)&=\E_x\left(\int_0^{\tau_{\ubar{b}^*(y)}^{\bar D}}e^{-\lambda(t)}dL_t^{\bar{b}^*(y)}+e^{-\lambda(\tau_{\ubar{b}^*(y)}^{\bar D})}J(X_{\tau_{\ubar{b}^*(y)}^{\bar D}})\right)\\
    &=\E_x\left(\int_0^{\tau_{\ubar{b}^*(y)}^{\bar D}}e^{-\lambda(t)}dD_t+e^{-\lambda(\tau_{\ubar{b}^*(y)}^{\bar D})}J(X_{\tau_{\ubar{b}^*(y)}^{\bar D}})\right)=J(x).
\end{align*}

\vspace{1mm}
{\it Step 1(e). Suppose that $x\in (\bar{b}^*(y),\infty)$.} 
This is analogous to Step 1.(c). 

\vspace{1mm}
{\it Step 2. Verification}.
Thanks to Step 1, we have that $J(\cdot;y,D^{b_{r+q}^*,\ubar{b}^*(y),\bar{b}^*(y)}) \in \mathcal{C}^2((0,\ubar{b}^*(y)) \cup (\ubar{b}^*(y)),y) \cup (y,\infty)) \cap \mathcal{C}^1(0,\infty)$.
Therefore, in order to apply Theorem \ref{verification_theorem}, it remains to prove that $J(\cdot;y,D^{b_{r+q}^*,\ubar{b}^*(y),\bar{b}^*(y)}) = w(\cdot;y)$ satisfies conditions \eqref{mainthmcond1}-\eqref{mainthmcond2}, which take the form
\begin{align}
w'(x;y) &\geq 1, \quad x\in (0,\infty), \label{mainthmcond1cr}  \tag{I'} \\
g(x;y) := \mathcal{L} w(x;y) \, I_{\{x\not=\ubar{b}^*(y),y\}} -\big(r + q \, I_{\{x<y\}} \big) \, w(x;y) &\leq 0, \quad x\in (0,\infty). \label{mainthmcond2cr}  \tag{II'} 
\end{align}
The proof is split in the following four steps.

\vspace{1mm}
{\it Step 2(a). Conditions \eqref{mainthmcond1cr}--\eqref{mainthmcond2cr} for $x\in (0,\ubar{b}^*(y)]$.} 
Since $y>\ubar{b}^*(y)$ and $w(\cdot;y) = V_{r+q}(\cdot)$, we can use the fact that $V_{r+q}$ satisfies the conditions \eqref{mainthmcond1cr}-- \eqref{mainthmcond2cr} (cf.~proof of Theorem \ref{theorem_optimal_yupper}). 
In particular, we have from \eqref{yupper_def} that
\begin{align} \label{wfefeef}
\mathcal{L}V_{r+q}(x) - (r+q) V_{r+q}(x) 
= \mu - (r+q) V_{r+q}(x)
< 0,  
\quad x \in (b^*_{r+q}, \ubar{b}^*(y)] \subset (b^*_{r+q}, y].
\end{align}

\vspace{1mm}
{\it Step 2(b). Condition \eqref{mainthmcond1cr} for $x\in (\ubar{b}^*(y),\bar{b}^*(y))$}. 
Using \eqref{wfefeef}, the ODE  \eqref{FBP_medyc} and the regularity of $w$ from \eqref{FBP_medye}, in particular $1 = V'_{r+q}(\ubar{b}^*(y)) = w'(\ubar{b}^*(y);y)$, we have 
\begin{align} \label{eq_help1}
\tfrac{\sigma^2}{2} w''(\ubar{b}^*(y)+,y) 
= - \mu + (r+q) w(\ubar{b}^*(y);y) 
= - \mu + (r+q) V_{r+q}(\ubar{b}^*(y)) > 0.
\end{align}
Using Lemma \ref{Lemma_ODE_Sol}.(ii), we observe that $w$ cannot have an inflection point in $(\ubar{b}^*(y),y)$, because it is convex around $\ubar{b}^*(y)$ by \eqref{eq_help1}. 
This implies that $w(\cdot;y)$ is convex on $(\ubar{b}^*(y),y)$, which together with $w'(\ubar{b}^*(y),y)=1$ yields that $w'(x,y)\geq 1$ for all $x\in (\ubar{b}^*(y),y]$. 
Then, using the fact that $w''(\bar{b}^*(y)-,y)=0$ thanks to \eqref{FBP_medyd}--\eqref{FBP_medye}, we can conclude again from Lemma \ref{Lemma_ODE_Sol}.(ii) that $w(\cdot;y)$ is concave on $(y,\bar{b}^*(y))$. 
Combining this with $w'(\bar{b}^*(y),y)=1$ yields that $w'(x,y)\geq 1$ for $x \in (y,\bar{b}^*(y))$. 

\vspace{1mm}
{\it Step 2(c). Condition \eqref{mainthmcond2cr} for $x\in (\ubar{b}^*(y),\bar{b}^*(y))$}. 
This is a straightforward consequence of the construction of $w(\cdot;y)$, which implies that $g(x;y)=0$ for all $x \in (\ubar{b}^*(y),\bar{b}^*(y))$.

\vspace{1mm}
{\it Step 2(d). Condition \eqref{mainthmcond1cr} for $x\in [\bar{b}^*(y),\infty)$.} 
This is a straightforward consequence of the construction of $w(\cdot;y)$, which implies that $w'(x;y)=1$ for all $x \in [\bar{b}^*(y),\infty)$.

\vspace{1mm}
{\it Step 2(e). Condition \eqref{mainthmcond2cr} for $x\in [\bar{b}^*(y),\infty)$.} 
Using the fact that $y < \bar{b}^*(y)$ and \eqref{FBP_medyd} we have  
\begin{align*}
g(x;y) = \mu - r w(x;y) 
\quad \text{and} \quad
g'(x;y) = - r < 0, 
\quad x\in (\bar{b}^*(y),\infty)
\end{align*}
Combining this with $g(\bar{b}^*(y);y)=0$, which holds true thanks to $w(x;y) \in \mathcal{C}^2$ at $x=\bar{b}^*(y)$ and \eqref{FBP_medyc}, 
yields that $g(x;y) \leq 0$ for all $x\in [b^*(y),\infty)$.

\end{proof}
We have thus shown that the strategy with two separate action regions in optimal. We conclude the chapter by showing that this strategy does indeed transition to the other cases as $y\downarrow y_l$ or $y\uparrow y_u$. This is shown in the following proposition which is proved in the Appendix \ref{App_critical}.

\begin{proposition} \label{prop_transition}
Let $(\ubar{b}^*(y),\bar{b}^*(y))$ be the optimal pair of boundaries in the critical regime given by Lemma \ref{fbp2_ex} and $b^*(y)$ the optimal boundary in the subcritical regime given by \eqref{def_deltay}. 
Then, we have 
\begin{enumerate}[\rm (i)]
\item $\underline{b}^*(y) - y \to 0 $ and $\, \overline{b}^*(y) -y \to 0$, as $y \uparrow y_u$;
\item $\underline{b}^*(y) \to b^*_{r+q}$ and $\, \overline{b}^*(y) \to b^*(y_l)$, as $y \downarrow y_l$.
\end{enumerate}
\end{proposition}

This result illustrates how the optimal dividend boundaries transition across regimes. 
Proposition~\ref{prop_transition}.(ii) shows that, as $y \downarrow y_l$, the upper boundary $\underline{b}^*(y)$ of the lower payout region $[b^*_{r+q}, \underline{b}^*(y)]$ and the lower boundary $\overline{b}^*(y)$ of the upper payout region $[\overline{b}^*(y), \infty)$ converge to the subcritical regime boundaries $b^*_{r+q}$ and $b^*(y_l)$, respectively, reflecting how the lower action region $[b^*_{r+q}, \underline{b}^*(y)]$ shrinks and the control approaches the classical subcritical strategy. Conversely, Proposition~\ref{prop_transition}.(i) shows that as $y \uparrow y_u$, both boundaries $\underline{b}^*(y)$ and $\overline{b}^*(y)$ converge to the distress threshold $y$, so that the two payout regions $[b^*_{r+q}, \underline{b}^*(y)]$ and $[\overline{b}^*(y), \infty)$ merge into a single contiguous action region. 
These results show that the optimal dividend strategy evolves continuously as the distress threshold $y$ changes: the boundaries of the payout regions adjust smoothly, reflecting how the firm balances the timing and size of dividend payments in response to the increasing or decreasing risk of occupation-time-based default.

\begin{figure}
     \centering
     \begin{subfigure}{0.3\textwidth}
         \centering
         \includegraphics[width=\linewidth]{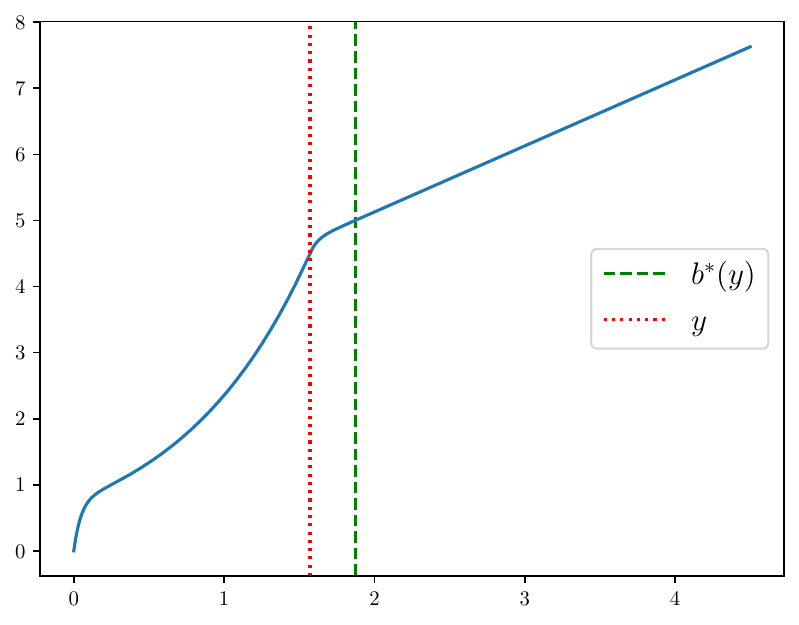}
     \end{subfigure}
     \begin{subfigure}{0.3\textwidth}
         \centering
         \includegraphics[width=\linewidth]{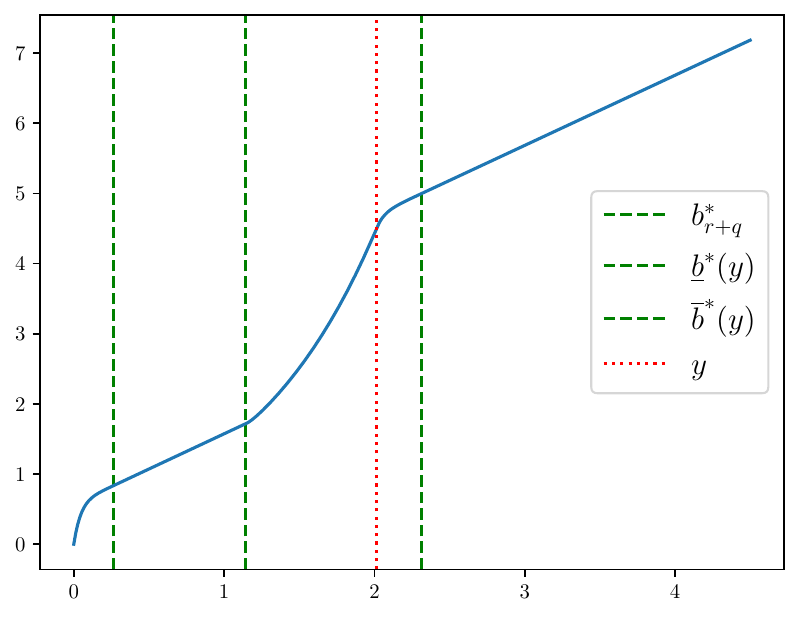}
     \end{subfigure}
     \begin{subfigure}{0.3\textwidth}
         \centering
         \includegraphics[width=\linewidth]{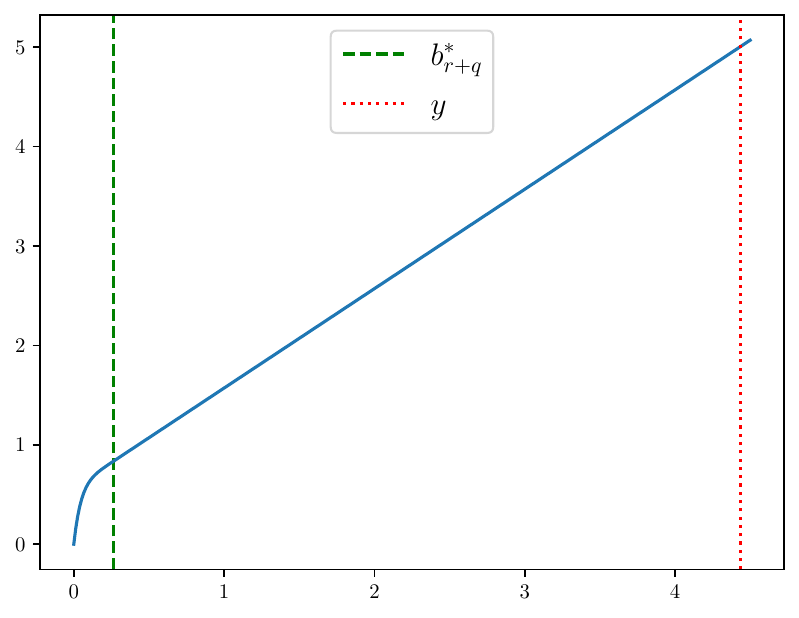}
     \end{subfigure}
     \begin{subfigure}{0.3\textwidth}
         \centering
         \includegraphics[width=\linewidth]{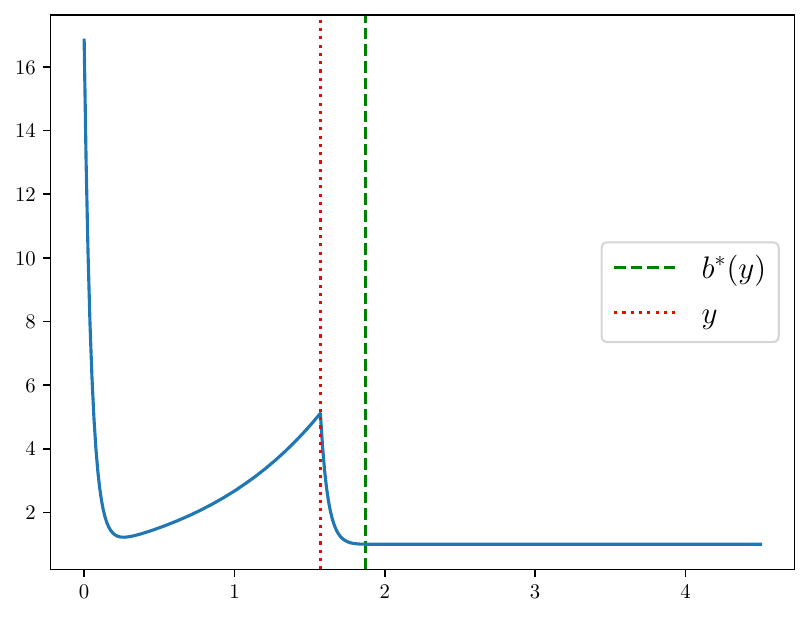}
     \end{subfigure}
     \begin{subfigure}{0.3\textwidth}
         \centering
         \includegraphics[width=\linewidth]{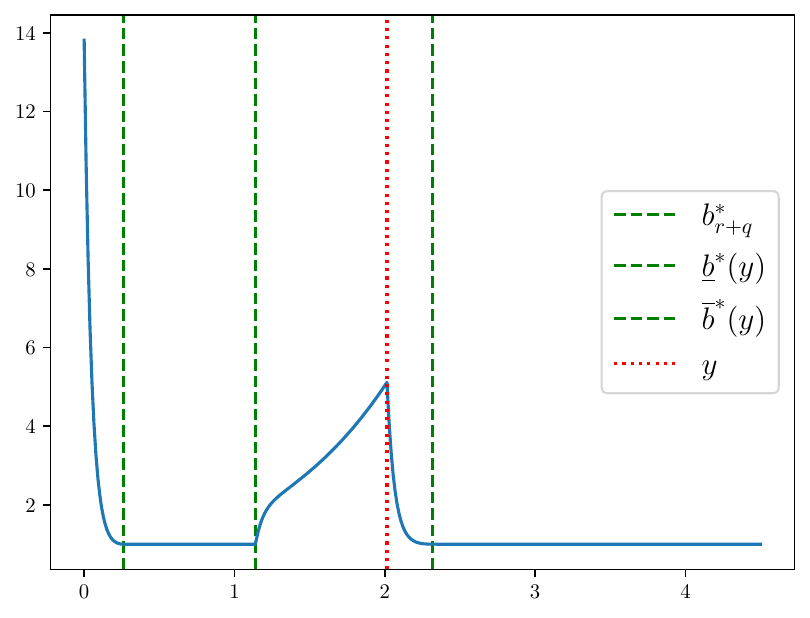}
     \end{subfigure}
     \begin{subfigure}{0.3\textwidth}
         \centering
         \includegraphics[width=\linewidth]{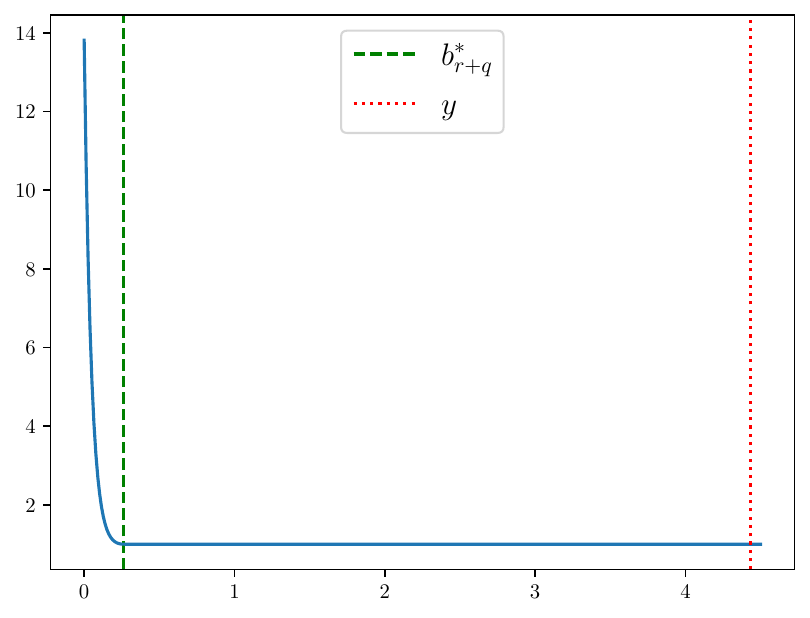}
     \end{subfigure}
          \caption{Value functions $V$ (top) and first derivatives $V'$ (bottom) for the three different regimes and the parameters $\mu =  0.1, \sigma =  0.1, r =  0.02, q =  0.1$. The three regimes result from different values of $y$ and are characterised by the pair $(y_l, y_u) \approx (1.7461, 4.4292)$. 
          On the left, we have $y = y_l\cdot0.9\leq y_l$ and the optimal dividend policy is to pay out dividends above the barrier $b^*(y)\approx 1.8717$. 
          In the middle, we have $y = 0.9 \cdot y_l + 0.1 \cdot y_u \in (y_l,y_u)$ and the optimal dividend policy is to pay out dividends in the disjoint set $[b^*_{r+q},\ubar{b}^*(y)]\cup[\bar{b}^*(y),\infty)\approx[0.2626,1.1394]\cup[ 2.3147,\infty)$. 
          On the right we have $y = y_u \cdot 1.001\geq y_u$ and the optimal dividend policy is to pay out dividends above $b^*_{r+q} \approx0.2626$, as in the classical case with ruin at zero and constant discount rate $r+q$.}
          \label{fig:1}
     \label{fig:1}
\end{figure}

\appendix
\section{Construction of a process with multiple jumps and reflections}
\label{sec:construction}

The aim of this section we recall the well known results of Skorokhod reflection as well as the construction of a process with multiple jump regions and reflections. 

\subsection{Connected action region: Skorokhod reflection strategy}
\label{scorohod_sec}

In this section, we recall the well known result of the Skorokhod reflection problem (see, e.g.~\cite{pilipenko2014introduction, Chitashvili01121981}). 
Given any threshold $b>0$, there exists a unique $\left({\cal F}_t\right)$-adapted pair $(X^{D^b},D^b)$ such that
\begin{align}\label{reflected_eq}
X^{D^b}_t = x + \mu t + \sigma W_t -D^b_t, \quad t\geq 0, \quad X^{D^b}_{0-}=x\geq 0,
\end{align}
and satisfies the following conditions
\begin{align} \label{SRP}
\begin{cases}
X^{D^b}_t\leq b, \quad t\geq 0, \quad \Pr-a.s., \\
D^b_t \;\;\text{is increasing} , \quad t\geq 0, \\
\int_0^t I_{\{X^{D^b}_s<b\}} dD_s^b = 0, \quad t\geq 0.
\end{cases}
\end{align}
By construction, the process $X^{D^b}$ started from $x \leq b$ is thus reflected downwards at $b$, whereas if it starts at $x>b$ it firsts jumps to $b$ (immediately at time $0$) with subsequently reflected at $b$. It can be shown that the process $D^b$ admits the form
\begin{align} \label{single_barrier_control}
D^b_t := (x-b)^+ + L^b_t(X^{D^b}), \quad t\geq 0, \quad D^b_{0-}=0,
\end{align} where $L^b(X^{D^b})$ is the (symmetric) local time of $X^{D^b}$ at the level $b$ defined by the
limit in probability
\begin{align*}
 L^b_t(X^{D^b}):= \lim_{\epsilon \to 0}  \frac{1}{2\epsilon}\int_0^t \sigma^2I_{\{b-\epsilon\leq X^{D^b}_t\leq b+\epsilon\}}dt.
\end{align*}
Hence we see that $D^b$ defined by \eqref{single_barrier_control} is adapted to $(\mathcal{F}_t)$ and thus $D^b$ is admissible, i.e.~$D^b\in \mathcal{A}$ from \eqref{adm}.

\subsection{Disconnected action region: Double barrier strategies}
\label{scorohod_sec_jump}

In this section, we proceed with a generalisation of the traditional Skorokhod reflection strategy (cf.~Section \ref{scorohod_sec}) which features an action region and inaction region with two disconnected components each. 
Such a control strategy $D^{b_1,c_1,b_2}$ involves an action region $[b_1,c_1] \cup [b_2,\infty)$ with two reflection boundaries $b_1,b_2$ for the controlled process $X^{D^{b_1,c_1,b_2}}$, such that $ 0 < b_1 < c_1 < b_2$ and the pair $(X^{D^{b_1,c_1,b_2}},D^{b_1,c_1,b_2})$ satisfies the SDE 
\begin{align} \label{reflected_eq_jump}
\begin{split}
&dX^{D^{b_1,c_1,b_2}}_t = \mu dt + \sigma dW_t - dD^{b_1,c_1,b_2}_t, \quad t\geq 0, \quad X^{D^{b_1,c_1,b_2}}_{0-} = x \geq 0, \quad D^{b_1,c_1,b_2}_{0-} = 0,
\end{split}
\end{align}
and the following conditions
\begin{align} \label{DRP}
\begin{cases}
X^{b_1,c_1,b_2}_t\in (-\infty,b_1]\cup(c_1,b_2], \quad t\geq 0, \quad \Pr-a.s. \\
D^{b_1,c_1,b_2}_t \;\;\text{is increasing} , \quad t\geq 0, \\
\int_0^t I_{\{X^{b_1,c_1,b_2}_t\in (-\infty,b_1)\cup(c_1,b_2)\}} dD_s^{b_1,c_1,b_2} = 0, \quad t\geq 0.
\end{cases}
\end{align}
We remark that a strong solution to \eqref{reflected_eq_jump} satisfying \eqref{DRP} exists, and in particular we have $D^{b_1,c_1,b_2} \in \mathcal{A}$ from \eqref{adm}. 
Its construction can be done in different ways (see e.g.~\cite{junca2024optimal}); however, we construct the process below for completeness.

We take an arbitrary $0<b_1<c_1<b_2$ and wish to construct a process that is reflected downwards at $b_2$ (with a possible jump to $b_2$ at time $0$ if $x>b_2$) until it reaches the interval $(b_1c_1]$, at which time it jumps to $b_1$, and then is subsequently reflected at $b_1$. 
To that end, we consider a concatenation of two SDEs using similar deductions as in \cite[Section 4]{LIU2025128809}. 

Firstly, consider the strong solution $X^{(1)}$ of the SDE (cf.~Section~\ref{scorohod_sec})
\begin{align*}
dX_t^{(1)} 
&= \mu dt + \sigma dW_t - dD^{(1)}_t, \qquad \qquad t\geq 0, \quad X_{0-}^{(1)}=x\geq 0, \\
D^{(1)}_t 
&= (x-b_2)I_{\{x>b_2\}} + L^{b_2}_t(X^{(1)}), \quad t\geq 0,
\quad D^{(1)}_{0-}=0,
\end{align*}
and define $\tau_{c_1}:=\inf\{t\geq 0:X^{(1)}_{t}\leq c_1\}$, which is finite  $\mathbb{P}$-a.s..

Then, define the Brownian motion $\widetilde W_t=W_{\tau_{c_1}+t}-W_{\tau_{c_1}}$ adapted to the filtration $\{\mathcal{F}_{t+\tau_{c_1}}\}_{t\geq 0}$ and consider the strong solution $X^{(2)}$ of the SDE (cf.~Section~\ref{scorohod_sec})
\begin{align*}
dX_t^{(2)} 
&= \mu dt + \sigma d\widetilde{W}_t -dD^{(2)}_t, \qquad \quad\, t\geq 0, \quad X_{0-}^{(2)} = X^{(1)}_{\tau_{c_1}} > 0, \\
D^{(2)}_t 
&= (X^{(1)}_{\tau_{c_1}} - b_1)^+ + L^{b_1}_t(X^{(2)}), \quad t\geq 0, \quad D^{(2)}_{0-}= 0.
\end{align*}

We are now ready to construct the process 
\begin{align*}
X_t &:= X^{(1)}_t I_{\{t<\tau_{c_1}\}} + X^{(2)}_{t-\tau_{c_1}} I_{\{t\geq\tau_{c_1}\}} 
\end{align*}
which satisfies the SDE 
\begin{align*} 
dX_t &= \mu dt + \sigma dW_t - dD^{b_1,c_1,b_2}_t, \quad t\geq 0, \quad X_{0-} = x \geq 0, \\
D^{b_1,c_1,b_2}_t 
&:= D^{(1)}_{t \wedge \tau_{c_1}} + D^{(2)}_{(t-\tau_{c_1})^+}, \qquad \quad\, t\geq 0, \quad D^{b_1,c_1,b_2}_{0-} = 0.
\end{align*}
This identifies with the SDE \eqref{reflected_eq_jump} and satisfies all conditions in \eqref{DRP} as well. 

\section{Proofs of results in Section \ref{chapter_classical_case}}
\label{App_classical}

\begin{proof}[Proof of Lemma 
\ref{lemma_useful_results}] 
We first recall the expression of $b_\rho^*$ in \eqref{value_classical} and observe from \eqref{gamma_def} that the numerator of the expression is positive and the denominator is positive and strictly increasing in $\rho$. 
Therefore, it suffices to show that 
$$
\rho \mapsto h(\rho) := \frac{\gamma_2^2(\rho)}{\gamma_1^2(\rho)} 
\quad \text{is strictly decreasing}. 
$$
Using \eqref{lemma_useful_results_gamma} we can show that $h(\rho) = (\gamma_2(\rho)\frac{\mu}{\rho}-1 )^2$, which gives
\begin{align*}
h'(\rho)
&=2\left(\gamma_2(\rho)\frac{\mu}{\rho}-1\right)\left(\gamma_2'(\rho)\frac{\mu}{\rho}-\gamma_2(\rho)\frac{\mu}{\rho^2}\right) 
=2\frac{\gamma_2(\rho)}{\gamma_1(\rho)}\frac{\mu}{\rho}\left(\gamma_2'(\rho)-\frac{\gamma_2(\rho)}{\rho}\right) < 0.
\end{align*}
To see the latter inequality, recall that $\gamma_1(\rho), \mu, \rho > 0$, $\gamma_2(\rho) < 0$, and by using the expressions in \eqref{gamma_def} that 
\begin{align}
\gamma_2'(\rho)-\frac{\gamma_2(\rho)}{\rho} 
&=-\tfrac{1}{\sigma^2\sqrt{\frac{\mu^2}{\sigma^4}+\frac{2\rho}{\sigma^2}}} + \tfrac{\sqrt{\frac{\mu^2}{\sigma^4}+\frac{2\rho}{\sigma^2}}+\frac{\mu}{\sigma^2}}{\rho}
=\tfrac{\frac{\mu^2}{\sigma^2\rho}+1}{\sigma^2\sqrt{\frac{\mu^2}{\sigma^4}+\frac{2\rho}{\sigma^2}}}+\frac{\mu}{\sigma^2\rho}>0,
\end{align}
which proves the desired statement.
\end{proof}

\section{Proofs of the Section \ref{sec_subcritical} (Subcritical regime)}
\label{App_sec_subcritical}

\begin{proof}[Proof of Lemma \ref{sol:FBP1}]
A fundamental solution to the ODE \eqref{FBP_lowya} is of the form 
\begin{align*}
w(x;y)=\begin{cases}
K_1(y) \, e^{\gamma_1(r+q) \, x} + K_2(y) \, e^{\gamma_2(r+q) \, x}, \quad &x\in (0,y),\\
K_3(y) \, e^{\gamma_1(r) \, x} + K_4(y) \, e^{\gamma_2(r) \, x}, \quad &x\in (y,b^*(y)),
\end{cases}
\end{align*}
for $\gamma_i(\cdot),i=1,2$, defined by \eqref{gamma_def} and some constants $K_i(y),i=1,\dots,4$, determined such that $v$ satisfies the boundary and regularity conditions \eqref{FBP_lowyb}--\eqref{FBP_lowyd}, namely via the system of equations
\begin{align} 
\label{lowy_lineadereq}
w(0+;y) &= 0, \quad 
w(y-;y) = w(y+;y), \quad 
w'(y-;y) = w'(y+;y), \quad 
w'(b^*(y)-;y) = 1.
\end{align}
Solving the above system of four linear equations, we obtain the functions $K_1(y), \ldots, K_4(y)$ defined by \eqref{constants_simple}.
By observing that $\delta'(y), \delta(y), \phi_r(y) >0$, as well as $\phi'_r(y) <0$, for all $y>0$, it can be shown that 
\begin{align*}
&(\delta\phi'_r-\delta'\phi_r)(y)<0, \quad y>0, \\
&(\delta\psi'_r-\delta'\psi_r)(y) = e^{(\gamma_1(r+q)+\gamma_1(r))y} (\gamma_1(r) - \gamma_1(r+q)) - e^{(\gamma_2(r+q)+\gamma_1(r))y} (\gamma_1(r) - \gamma_2(r+q)) < 0, \quad y>0. 
\end{align*}
Using the above inequalities we can therefore see that 
\begin{align*}
\delta(y) \eta'(y; b) - \delta'(y) \eta(y; b)=\psi'_{r}(b)(\delta\phi'_r-\delta'\phi_r)(y)-\phi'_{r}(b)(\delta\psi'_r-\delta'\psi_r)(y)<0, \quad y>0, 
\end{align*}
implying that all the denominators of the constants in \eqref{constants_simple} are non-zero and hence well-defined for all pairs $(y,b) \in (0,\infty)^2$. 

Then, by using the $\mathcal{C}^2$-regularity of $w(\cdot;y)$ at $b^*(y)$ from \eqref{FBP_lowyd} for all $y>0$, we have $w''(b^*(y)-; y) = 0$, which thus yields that $b^*(y)$ is given as a solution to the implicit equation (due to the dependence of $K_3(y)$, $K_4(y)$ on $b^*(y)$)
\begin{align}\label{b_star_simple1}
b^*(y) 
= \frac{1}{\gamma_1(r) - \gamma_2(r)} \log\left(-\frac{K_4 ( y) \gamma_2^2(r)}{K_3( y) \gamma_1^2(r)} \right).
\end{align}
However, straightforward calculations involving the expressions \eqref{constants_simple} of $K_3(y)$ and $K_4(y)$ yield that  
\begin{align}\label{eq:C4divC3}
\hspace{-1mm}
-\frac{K_4(y)}{K_3(y)}
= \frac{\delta\psi'_r - \delta'\psi_r}{\delta \phi'_r - \delta' \phi_r}(y) 
= e^{(\gamma_1(r) - \gamma_2(r))y} \frac{e^{\gamma_1(r+q)y}\bigl(\gamma_1(r+q) - \gamma_1(r)\bigr) + e^{\gamma_2(r+q)y}\bigl(\gamma_1(r) - \gamma_2(r+q)\bigr)}{e^{\gamma_1(r+q)y}\bigl(\gamma_1(r+q) - \gamma_2(r)\bigr) + e^{\gamma_2(r+q)y}\bigl(\gamma_2(r) - \gamma_2(r + q)\bigr)},
\end{align}
which is eventually independent of $b^*(y)$ and positive. 
Therefore, plugging this into \eqref{b_star_simple1} yields the expression of $b^*(y)$ stated in Lemma \ref{sol:FBP1}. 
The remaining assertions that 
$\Delta(y) > 0$ for $y\in (0,y_u]$ and $\frac{d}{dy} b^*(y) > 0$ for all $y > 0$ follow from Lemma \ref{lemma:yl}.(iv). 
\end{proof}

\begin{proof}[Proof of Lemma \ref{lemma_yl}]
We prove the result in two parts. 

\vspace{1mm}
{\it Part} (I). We first aim at proving, the existence of a unique $y_l\in (0,y_u)$, satisfying $y_l\in (b^*_{r+q},y_u)$ and \eqref{ylower_def}. 
To that end, recall the expression of $w'(b^*_{r+q}; y)$ from \eqref{w_sol_simple_def}--\eqref{def_deltay}, which can take two possible forms, depending on whether $b^*_{r+q} < y < b^*(y)$ or $y \leq b^*_{r+q} < b^*(y)$.
Note that the inequality $b^*_{r+q} < b^*(y)$ follows from Lemma \ref{lemma:yl}.(i) for all $0 < y \leq y_u$. 

Then, notice from Lemma \ref{sol:FBP1} that 
$w'(b^*_{r+q}; \cdot) \in \mathcal{C}((0, y_u]) \cap \mathcal{C}^1((0, b^*_{r+q}) \cup (b^*_{r+q}, y_u))$.
Thanks to this, the existence of a unique $y_l \in (0,y_u)$ such that $w'(b^*_{r+q}; y)=1$ follows by showing that $w'(b^*_{r+q}; y) > 1$ for all $y \in (0, b^*_{r+q}]$, and that $y \mapsto w'(b^*_{r+q}; y)$ is strictly decreasing for $y \in (b^*_{r+q}, y_u]$ with $w'(b^*_{r+q}; y_u)<1$ (by invoking the intermediate value theorem). 
These properties are proved in the following three steps.

\vspace{1mm}
{\it Step 1. Proof of $w'(b^*_{r+q}; y) > 1$ for $y \in (0, b^*_{r+q}]$.}
We begin by noting that $w''(x; y) < 0$ for all $x \in (y, b^*(y))$, thanks to 
\begin{align} \label{help_ineq}
w''(x; y) 
= \gamma_1^2(r) K_3(y)e^{\gamma_1(r)x} + \gamma_2^2(r) K_4(y)e^{\gamma_2(r)x}
< 0 
\; \Leftrightarrow \;
H(x):=-\frac{K_4(y)\gamma_2^2(r)}{K_3(y)\gamma_1^2(r)}e^{(\gamma_2(r) - \gamma_1(r))x} > 1.
\end{align}
To see the latter inequality, observe that $H(b^*(y))=1$ and $H'(x) < 0$, since $-{K_4(y)}/{K_3(y)} > 0$ thanks to Lemma \ref{lemma:yl}.(v).
Therefore, $w'(\cdot;y)$ is strictly decreasing on $(y, b^*(y))$, and given that $w'(b^*(y); y) = 1$ thanks to Lemma \ref{sol:FBP1}, we can conclude that $w'(x; y) > 1$ for all $x \in [y, b^*(y))$. 
In particular, using the fact that $b^*_{r+q} \in [y, b^*(y))$ in this case, we get the desired property. 

\vspace{1mm}
{\it Step 2. Proof of $\frac{\partial}{\partial y} w'(b^*_{r+q}; y) < 0$ for $y \in (b^*_{r+q}, y_u]$.}
First recall that by \eqref{value_classical}, Lemma \ref{sol:FBP1} for $b^*_{r+q} < y$, and \eqref{def:delta}, we have
\begin{align}
w'(b^*_{r+q}; y) 
&= \gamma_1(r+q)K_1(y) e^{\gamma_1(r+q)b^*_{r+q}} + \gamma_2(r+q)K_2(y) e^{\gamma_2(r+q)b^*_{r+q}} \notag\\
&= K_1(y) \bigg(\gamma_1(r+q) \left( \frac{\gamma_2^2(r+q)}{\gamma_1^2(r+q)}\right)^{\frac{\gamma_1(r+q)}{\gamma_1(r+q) - \gamma_2(r+q)}}
- \gamma_2(r+q) \left( \frac{\gamma_2^2(r+q)}{\gamma_1^2(r+q)} \right)^{\frac{\gamma_2(r+q)}{\gamma_1(r+q) - \gamma_2(r+q)}} \bigg) 
\notag\\
&= K_1(y) \, \delta'(b^*_{r+q}) 
\label{eq:C_1neu},
\end{align}
where $\delta'(b^*_{r+q}) > 0$ is independent of $y$. 
The property then follows from Lemma \ref{lemma:yl}.(vi).

\vspace{1mm}
{\it Step 3. Proof of $w'(b^*_{r+q};y_u)<1$.}
Substituting $y=y_u$ in \eqref{eq:C_1neu} and using Lemma \ref{lemma:yl}.(vi), we see that
\begin{align*}
w'(b^*_{r+q}; y_u) = K_1(y_u) \, \delta'(b^*_{r+q}) < 1.
\end{align*}

\vspace{1mm}
{\it Part} (II).
Now, we aim at proving that $y_l$ is the unique solution to the equation $f(y) = 0$ where $f:(0,y_u) \to \R$ is defined by \eqref{ylower_eq_charac}. 
To that end, we use the fact that $y_l\in(b^*_{r+q},y_u)$ and the continuity of $w'(b^*_{r+q}; \cdot)$, to conclude that 
\begin{align} \label{w'b=1}
w'(b^*_{r+q}; y_l) = K_1(y_l) \, \delta'(b^*_{r+q}) = 1 ,
\end{align}
which together with the explicit form of $K_1$ in Lemma \ref{lemma:yl}.(vi) provides the desired characterisation.
\end{proof}

\begin{lemma}\label{lemma:yl}
Recall $b^*_{r+q}$ from \eqref{value_classical}, $\delta$ from \eqref{def:delta}, $K_1, \ldots, K_4$ from \eqref{constants_simple}, and $b^*, \Delta$ from \eqref{def_deltay}. 
Then, we have:
\begin{enumerate}[\rm (i)]
\item 
$b^*_{r+q} < b^*_{r} < b^*(y)$ for all $y > 0$;
\item 
$b^*_{r+q}$ is the unique $y$-value that satisfies $\delta''(y)=0$, 
and 
$\delta''(y) > 0$ if and only if $y > b^*_{r+q}$;
\item 
$\frac{\delta'(x)}{\delta'(b^*_{r+q})}>1$ for all $x\in (b^*_{r+q},\infty)$ 
and 
$\frac{\delta(y_u)}{\delta'(y_u)}< \frac{\mu}{r} < \frac{\delta(y_u)}{\delta'(b^*_{r+q})}$;
\item $\Delta(y_u)>0$ and $-1 < \Delta'(y) < 0$ for all $y > 0$;
\item $K_3(y) > 0$ and $K_4(y) < 0$ for all $y > 0$;
\item $K_1(y)$ admits the form
\begin{align}\label{k1_form} 
K_1(y)= \left( \left(\frac{\gamma_2^2(r)}{\gamma_1^2(r)}\frac{\delta'(y) - \gamma_1(r)\delta(y)}{\delta'(y) - \gamma_2(r)\delta(y)}  \right)^{\frac{\gamma_1(r)}{\gamma_1(r) - \gamma_2(r)}}
\left(\frac{\gamma_1(r)}{-\gamma_2(r)}\delta'(y) + \gamma_1(r)\delta(y) \right)\right)^{-1}
\end{align}
satisfying $K_1(y) >0$ for all $y > 0$, $K_1'(y) < 0$ for $b^*_{r+q} < y \leq y_u$, as well as $K_1(y_u)\delta'(b^*_{r+q})<1$.
\end{enumerate}
\end{lemma}

\begin{proof}[Proof of Lemma \ref{lemma:yl}]
We prove each part separately. 

\vspace{1mm}
{\it Proof of part} (i).
The first inequality $b^*_{r+q} < b^*_r$ is proved in Lemma \ref{lemma_useful_results}, thus it suffices to prove that $b^*_r < b^*(y)$ for all $y > 0$. 
To that end, we use their expressions in \eqref{value_classical} and \eqref{def_deltay} to get  
\begin{align}\label{eq:bstar05}
\begin{split}    
b^*(y) &:= y +\Delta(y) 
= \; y + b^*_r + \frac{1}{d(r)} \log\Big( \frac{c_1 e^{\gamma_1(r+q)y} + c_2 e^{\gamma_2(r+q)y}}{c_2 e^{\gamma_1(r+q)y} + c_1 e^{\gamma_2(r+q)y}} \Big),
\end{split}    
\end{align}
for the positive constants $c_1, c_2$ and the increasing function $d(\cdot)$ defined by 
\begin{align}\label{cd_def}
\begin{split}
0 < c_1 := \gamma_1(r+q) - \gamma_1(r) = \gamma_2(r) - \gamma_2(r + q) &< \gamma_1(r) - \gamma_2(r+q) = \gamma_1(r+q) - \gamma_2(r) =: c_2,\\
0 < d(r) := \gamma_1(r) - \gamma_2(r) = c_2 - c_1 &< c_2 + c_1 = \gamma_1(r+q) - \gamma_2(r+q) =: d(r+q),
\end{split}
\end{align}
where the equalities in the first line follow from \eqref{lemma_useful_results_gamma}.
Then, we observe from \eqref{eq:bstar05}--\eqref{cd_def} that
\begin{align*}
\begin{split}
b^*_r < b^*(y)
\quad \Leftrightarrow \quad 
L(y) := e^{- d(r) y} \, \frac{c_2 e^{d(r+q)y} + c_1}{c_1 e^{d(r+q)y} + c_2} < 1 ,
\quad y>0 .
\end{split}    
\end{align*}
This holds true since $L(0)=1$ and $L'(y)<0$ for all $y>0$, which follows from \eqref{eq:bstar05} and 
\begin{align*}
\big(e^{d(r+q)y}c_1 + c_2 \big)^2 L'(y) 
&= - c_1 c_2 d(r) e^{-d(r)y} \big(e^{2d(r+q)y} + 1 \big) \\ 
&\quad + e^{(d(r+q)-d(r))y} \big(c_2^2 (d(r+q)-d(r)) -c_1^2 (d(r+q)+d(r)) \big) \\
&= -c_1 c_2 d(r) e^{-d(r)y} \big(e^{2d(r+q)y} - 2e^{d(r+q)y} + 1 \big) < 0.
\end{align*}
The latter positivity follows by observing that $\widetilde L(y):=e^{2d(r+q)y} - 2e^{d(r+q)y} + 1 > 0$ for all $y>0$, thanks to $\widetilde L(0)=0$ and $\widetilde L'(y) = 2d(r+q)\left(e^{2d(r+q)y} -e^{d(r+q)y}\right) > 0$ for $y > 0$.

\vspace{1mm}
{\it Proof of part} (ii).
Given that 
$\delta''(y) = \gamma_1^2(r+q)e^{\gamma_1(r+q)y} - \gamma_2^2(r+q)e^{\gamma_2(r+q)y}$, we can directly verify that $y = b^*_{r+q}$ is the only solution to $\delta''(y)=0$ and that $\delta''(y)>0$ for all $y > b^*_{r+q}$. 

\vspace{1mm}
{\it Proof of part} (iii).
First note that the following equalities hold true (cf.~\eqref{value_classical}, \eqref{yupper_def}) 
\begin{align*} 
V'_{r+q}(x)=1 \quad \text{for} \quad x \geq b^*_{r+q} 
\quad \text{with} \quad 
V_{r+q}(b^*_{r+q}) = \tfrac{\delta(b^*_{r+q})}{\delta'(b^*_{r+q})} 
\quad \text{and} \quad 
V_{r+q}(y_u)=\frac{\mu}{r}. 
\end{align*}
Combining the above with 
the fact that $\delta''(x)>0$, for all $x > b^*_{r+q}$, thanks to Lemma \ref{lemma:yl}.(ii), and the fundamental theorem of calculus, we obtain the desired results via 
\begin{align*}
\frac{\mu}{r} - \frac{\delta(y_u)}{\delta'(y_u)} 
&= V_{r+q}(y_u) - V_{r+q}(b^*_{r+q}) - \frac{\delta(y_u)}{\delta'(y_u)} + \frac{\delta(b^*_{r+q})}{\delta'(b^*_{r+q})} 
= \int_{b^*_{r+q}}^{y_u} \Big( V_{r+q}'(x) - 1 + \frac{\delta(x)\delta''(x)}{(\delta'(x))^2} \Big) dx > 0, \notag\\
\frac{\delta'(x)}{\delta'(b^*_{r+q})} - 1 &= \int_{b^*_{r+q}}^{x} \frac{\delta''(z)}{\delta'(b^*_{r+q})} dz > 0, \quad x>b^*_{r+q}, \\
\frac{\delta(y_u)}{\delta'(b^*_{r+q})} - \frac{\mu}{r} 
&= \frac{\delta(y_u)}{\delta'(b^*_{r+q})} - \frac{\delta(b^*_{r+q})}{\delta'(b^*_{r+q})} + V_{r+q}(b^*_{r+q}) - V_{r+q}(y_u) = \int_{b^*_{r+q}}^{y_u} \Big(\frac{\delta'(x)}{\delta'(b^*_{r+q})} - 1 \Big) dx > 0.
\end{align*}

\vspace{1mm}
{\it Proof of part} (iv).
Firstly, we recall from Lemma \ref{lemma:yl}.(iii) and the expression of $\gamma_1$ from \eqref{gamma_def} that
\begin{align} \label{ineq1}
\frac{\delta(y_u)}{\delta'(y_u)}<\frac{\mu}{r} < \frac1{\gamma_1(r)} 
\quad \Leftrightarrow \quad 
\delta'(y_u)-\gamma_1(r)\delta(y_u) > 0.
\end{align}
Combining this with the definition \eqref{def_deltay} of $\Delta(y)$ and the positivity of $\delta(y), \delta'(y)$ from \eqref{def:delta} for all $y>0$, together with the fact that $\gamma_2(r) < 0 < \gamma_1(r)$, we get that
\begin{align*}
&\Delta(y_u) 
= \frac{\log\left( \frac{(\delta'(y_u)-\gamma_1(r)\delta(y_u)) \gamma_2^2(r)}{(\delta'(y_u)-\gamma_2(r)\delta(y_u)) \gamma_1^2(r)} \right)}{\gamma_1(r)-\gamma_2(r)}  > 0 
\quad \Leftrightarrow \quad 
(\delta'(y_u)-\gamma_1(r)\delta(y_u))\gamma_2^2(r)>(\delta'(y_u)-\gamma_2(r)\delta(y_u))\gamma_1^2(r).
\end{align*}
To see how the latter inequality holds true, divide both sides by 
$-\gamma_1(r) \gamma_2(r) \delta'(y_u) > 0$, and observe that it is equivalent to
\begin{align*}
\frac{\gamma_2^2(r)-\gamma_1^2(r)}{-\gamma_1(r)\gamma_2(r)} - \frac{\delta(y_u)}{\delta'(y_u)}(\gamma_1(r)-\gamma_2(r)) >0,
\end{align*}
which holds true thanks to $\frac{\delta(y_u)}{\delta'(y_u)}<\frac{\mu}{r}$ from \eqref{ineq1} (see also Lemma~\ref{lemma:yl}.(iii)) and \eqref{lemma_useful_results_Vb}. 

For the monotonicity of $\Delta(\cdot)$ defined by \eqref{def_deltay}, 
which takes the form as in \eqref{eq:bstar05}
\begin{align*} 
\begin{split}
\Delta(y) =b^*_r+ \frac{1}{d(r)}\log \Big(\frac{u(y)}{v(y)} \Big),
\quad \text{where} \quad 
u(y):=c_1e^{\gamma_1(r+q)y} + c_2e^{\gamma_2(r+q)y},  
\quad 
v(y):=c_2e^{\gamma_1(r+q)y} + c_1e^{\gamma_2(r+q)y},
\end{split}
\end{align*}
for $0 < c_1 <  c_2 < \infty$, $0 < d(r) < d(r+q) < \infty$ given by \eqref{cd_def}, 
we calculate 
\begin{align} \label{eq:DeltaPrime1}
\Delta'(y) 
= \frac{u'(y) v(y) - u(y) v'(y)}{d(r)u(y)v(y)} < 0.
\end{align}
The latter inequality follows from the straightforward positivity of all three factors in the denominator, and the negativity of the numerator which results from 
\begin{align}\label{eq:DeltaPrime2}
\begin{split}
u'(y) v(y) - u(y) v'(y) 
&= \big(\gamma_1(r+q) c_1 e^{\gamma_1(r+q)y} + \gamma_2(r+q) c_2 e^{\gamma_2(r+q)y}\big) \big(c_2 e^{\gamma_1(r+q)y} + c_1 e^{\gamma_2(r+q)y} \big) \\ 
&\quad - \big(c_1 e^{\gamma_1(r+q)y} +  c_2 e^{\gamma_2(r+q)y} \big) \big(\gamma_1(r+q) c_2 e^{\gamma_1(r+q)y} + \gamma_2(r+q) c_1 e^{\gamma_2(r+q)y} \big) \\
&= e^{(\gamma_1(r+q) +  \gamma_2(r+q))y} \, d(r+q) \left(c_1^2 - c_2^2 \right) < 0 ,
\end{split}
\end{align}
thanks to $0 < c_1 < c_2$. 
To complete the monotonicity properties, it remains to show that $\Delta'(y) > -1$ for all $y > 0$. 
To that end, using \eqref{cd_def}, \eqref{eq:DeltaPrime1} and \eqref{eq:DeltaPrime2}, we get 
\begin{align*}
\Delta'(0) 
= \frac{d(r+q)(c_1^2 - c_2^2)}{d(r) (c_1 + c_2)^2} = -1,
\quad \text{since} \quad 
c_1^2 - c_2^2 = -d(r)d(r+q)
\quad \text{and} \quad 
(c_1+c_2)^2 = d(r+q)^2 .
\end{align*}
Combining this with $\Delta \in C^2([0, \infty))$ and \eqref{eq:DeltaPrime1}--\eqref{eq:DeltaPrime2}, we get
\begin{align*}
\Delta''(y) 
&= - d(r+q)^2 \frac{d}{dy} \Big(\frac{e^{(\gamma_1(r+q) + \gamma_2(r+q))y}}{u(y)v(y)} \Big) = d(r+q)^3 c_1 c_2 \frac{\big(e^{2\gamma_1(r+q)y} - e^{2\gamma_2(r+q)y} \big) e^{(\gamma_1(r+q)+\gamma_2(r+q))y}}{u^2(y) v^2(y)} > 0,
\end{align*}
which implies that $\Delta'(\cdot)$ is strictly increasing, thus $\Delta'(y) > -1$ for all $y > 0$.

\vspace{1mm}
{\it Proof of part} (v). 
Recall the definitions \eqref{constants_simple} of $K_3$ and $K_4$ and \eqref{cd_def} of the constants $0 < c_1 <  c_2 < \infty$, as well as the fact that $\gamma_2(r) < 0 < \gamma_1(r)$. 
Substituting the expression \eqref{def_deltay} of $b^*$ in \eqref{constants_simple} then gives
\begin{align*}
K_4(y) &= -\frac{e^{-\gamma_2(r)y} \big(e^{\gamma_1(r+q)y} c_1 + e^{\gamma_2(r+q)y} c_2 \big)
    }
    {
    e^{\gamma_1(r+q)y} \left( 
    e^{\gamma_1(r)\Delta(y)} c_2\gamma_1(r) - e^{\gamma_2(r)\Delta(y)} c_1\gamma_2(r)
    \right) 
    + e^{\gamma_2(r+q)y} \left(
    e^{\gamma_1(r)\Delta(y)} c_1\gamma_1(r) - e^{\gamma_2(r)\Delta(y)} c_2\gamma_2(r) \right)
    },
\end{align*}
which is negative for all $y > 0$. 
Combining this with the positivity of the fraction in \eqref{eq:C4divC3}, we can conclude that $K_3(y)$ has the opposite sign to $K_4(y)$, hence $K_3(y) > 0$ for $y > 0$.

\vspace{1mm}
{\it Proof of part} (vi). 
Using the definition of $K_1$ from \eqref{constants_simple}, we observe after long calculations that 
\begin{align}
K_1(y) &= \frac{\gamma_1(r)-\gamma_2(r)}{G(y)}, \quad \text{where} \quad 
\label{k_1_repg}\\ 
G(y)&:= \left( \gamma_1(r) e^{\gamma_1(r)\Delta(y)}  - \gamma_2(r) e^{\gamma_2(r)\Delta(y)}\right)\delta'(y) - \gamma_1(r)\gamma_2(r)\left(e^{\gamma_1(r) \Delta(y)} - e^{\gamma_2(r)\Delta(y)} \right) \delta(y) 
\label{eq:G_y}.
\end{align}
Using \eqref{def_deltay} then gives 
\begin{align}
G(y)
&=e^{\gamma_1(r)\Delta(y)}\left(\left( \gamma_1(r)  - \gamma_2(r) e^{(\gamma_2(r)-\gamma_1(r))\Delta(y)}\right)\delta'(y) - \gamma_1(r)\gamma_2(r)\left(1 - e^{(\gamma_2(r)-\gamma_1(r))\Delta(y)} \right) \delta(y)\right)\notag\\
&=e^{\gamma_1(r)\Delta(y)} \left(-\gamma_2(r)e^{(\gamma_2(r)-\gamma_1(r))\Delta(y)}\left(\delta'(y)-\gamma_1(r)\delta(y)\right)+\gamma_1(r)\delta'(y)-\gamma_1(r)\gamma_2(r)\delta(y)\right)\notag\\
&=e^{\gamma_1(r)\Delta(y)}\Big(-\frac{\gamma_1^2(r)}{\gamma_2(r)}(\delta'(y)-\gamma_2(r)\delta(y))+\gamma_1(r)\delta'(y)-\gamma_1(r)\gamma_2(r)\delta(y)\Big)\notag\\
&=e^{\gamma_1(r)\Delta(y)}(\gamma_1(r)-\gamma_2(r))\Big(\gamma_1(r)\delta(y)-\frac{\gamma_1(r)}{\gamma_2(r)}\delta'(y)\Big) > 0, \quad y > 0,
\label{g_rep1} 
\end{align}
whose positivity follows from \eqref{g_rep1} thanks to the fact that $\gamma_2(r) < 0 < \gamma_1(r)$ and $\delta(y), \delta'(y) > 0$ for $y>0$ by definition \eqref{def:delta}.
This provides the desired form of $K_1$ in \eqref{k1_form} upon substituting the expression of $\Delta$ from \eqref{def_deltay} into \eqref{g_rep1}.
Hence, the positivity of $K_1(y)$ for all $y>0$ follows from \eqref{k_1_repg} and \eqref{g_rep1}.

Regarding the monotonicity of $K_1$, we notice that $K_1'(y) = - (\gamma_1(r)-\gamma_2(r)) \frac{G'(y)}{G^2(y)}$, whose sign is the opposite of that of $G'(y)$. Taking the derivative of $G$ in \eqref{eq:G_y} gives
\begin{align} \label{gprime_eq}
G'(y)
&= \Delta'(y) \left[ \gamma_1^2(r)e^{\gamma_1(r)\Delta(y)}\left( \delta'(y) - \gamma_2(r) \delta(y) \right)
- \gamma_2^2(r)e^{\gamma_2(r)\Delta(y)}\left( \delta'(y) - \gamma_1(r)\delta(y) \right)
\right] \notag\\
&\quad -\gamma_1(r)\gamma_2(r) \delta'(y)\left(e^{\gamma_1(r)\Delta(y) } - e^{\gamma_2(r)\Delta(y)} \right)
+ \delta''(y)\left(\gamma_1(r)e^{\gamma_1(r)\Delta(y) } - \gamma_2(r)e^{\gamma_2(r)\Delta(y)} \right) \notag\\
&= -\gamma_1(r)\gamma_2(r) \delta'(y)\left(e^{\gamma_1(r)\Delta(y) } - e^{\gamma_2(r)\Delta(y)} \right)
+ \delta''(y)\left(\gamma_1(r)e^{\gamma_1(r)\Delta(y) } - \gamma_2(r)e^{\gamma_2(r)\Delta(y)} \right),
\end{align}
where the second equality follows from the definition \eqref{def_deltay} of $\Delta(y)$. 
To see the positivity of $G'(y)$ for all $y \in (b^*_{r+q}, y_u]$, observe in \eqref{gprime_eq} that $\Delta(y) > 0$ thanks to Lemma \ref{sol:FBP1}, $\delta'(y) > 0$ thanks to \eqref{def:delta}, and $\delta''(y) > 0$ by Lemma \ref{lemma:yl}.(ii). 

Finally, using \eqref{k_1_repg} together with the positivity and the explicit expression of $G$ from \eqref{g_rep1}, we notice that $K_1(y_u)\delta'(b^*_{r+q})<1$ is equivalent to 
\begin{align*}
\delta'(b^*_{r+q}) < \frac{G(y_u)}{\gamma_1(r)-\gamma_2(r)} = e^{\gamma_1(r)\Delta(y_u)} \left(-\frac{\gamma_1(r)}{\gamma_2(r)}\delta'(y_u)+\gamma_1(r)\delta(y_u)\right). 
\end{align*}
To prove this, given that $\gamma_1(r)>0$ and $\Delta(y_u) > 0$ by Lemma \ref{lemma:yl}.(iv), it is sufficient to show that
\begin{align*}
- \frac{\gamma_1(r) \delta'(y_u)}{\gamma_2(r) \delta'(b^*_{r+q})} + \gamma_1(r) \frac{\delta(y_u)}{\delta'(b^*_{r+q})} - 1 > 0.
\end{align*}
This follows from $\gamma_2(r)<0< \gamma_1(r)$ in \eqref{gamma_def}, Lemma \ref{lemma:yl}.(iii), \eqref{yupper_def} and \eqref{lemma_useful_results_gamma}, 
and completes the proof.
\end{proof}

\section{Proofs of Section \ref{sec_critical_regime} (Critical regime)}
\label{App_critical}

\begin{proof}[Proof of Lemma \ref{fbp2_ex}] 
It is clear that a fundamental solution $w(\cdot;y)$ to \eqref{FBP_medya}-\eqref{FBP_medyd} is of the form specified in \eqref{w_sol_double_def}, with $E_1(y), \ldots, E_4(y), \ubar{b}^*(y), \bar{b}^*(y)$ to be uniquely determined, such that the pair $(\ubar{b}^*(y),\bar{b}^*(y)) \in (b_{r+q}^*,y) \times (y,\infty)$.
We prove this in the following steps.

\vspace{1mm}
{\it Step 1}. We fix an arbitrary pair $(\ubar{b}(y),\bar{b}(y)) \in (b_{r+q}^*,y) \times (y,\infty)$ and use it in the place of the free-boundaries $\ubar{b}^*(y)$ and $\bar{b}^*(y)$.  
Then, by imposing the property that $w(\cdot;y)$ is $\mathcal{C}^1(0,\bar{b}(y))$, we obtain the following boundary value problem 
\begin{align}
\begin{split}
V_{r+q}(\ubar{b}(y)-) &= w(\ubar{b}(y)+;y), \qquad \qquad \; 
w(y-;y) = w(y+;y), \quad \\
V_{r+q}'(\ubar{b}(y)-) &= 1 = w'(\ubar{b}(y)+;y), \qquad 
w'(y-;y) = w'(y+;y). 
\end{split}
\label{medy_lineareq}
\end{align}
In order to solve the above boundary value problem, we first consider the functions $e_1(b), e_2(b), e_3(b,y), e_4(b,y)$ defined by \eqref{constants_complex}, which are well-defined for any arbitrary pair $(b,y)\in[0,\infty)^2$ since all denominators are positive. 
Then, we observe that the solution to the boundary value problem is given by choosing the functions $E_1, \ldots, E_4$ according to 
$E_1(y) = e_1(\ubar{b}(y)), 
E_2(y) = e_2(\ubar{b}(y)), 
E_3(y) = e_3(\ubar{b}(y),y), 
E_4(y) = e_4(\ubar{b}(y),y)$.

\vspace{1mm}
{\it Step 2}. 
Since we require the solution $w(\cdot;y) \in \mathcal{C}^2((0,\ubar{b}^*(y)) \cup (\ubar{b}^*(y),y) \cup (y,\infty)) \cap \mathcal{C}^1((0,\infty)$ (cf.~property \eqref{FBP_medye}), we further impose that $w(\cdot;y)$ is $\mathcal{C}^2(\{\bar{b}(y)\})$.
This implies that 
\begin{align}\label{v_optimal_eqwww}
w'(\bar{b}(y);y)=1 
\quad \text{and} \quad 
w''(\bar{b}(y);y)=0,
\end{align}
which yields the following two-dimensional system of equations 
\begin{align} \label{eq_system_midy}
\begin{split}
e_3(\ubar{b}(y),y) \psi'_r(\bar{b}(y)) + e_4(\ubar{b}(y),y) \phi'_r(\bar{b}(y)) &= 1,\\
e_3(\ubar{b}(y),y) \psi''_r(\bar{b}(y)) + e_4(\ubar{b}(y),y) \phi''_r(\bar{b}(y)) &=0.
\end{split}
\end{align}
Then, solving this system will yield the expressions of the free-boundaries $\ubar{b}^*(y)$ and $\bar{b}^*(y)$. 

To that end, we observe that the second equation in \eqref{eq_system_midy} gives 
\begin{align*}
\frac{\psi''_r(\bar{b}^*(y))}{\phi''_r(\bar{b}^*(y))} = -\frac{e_4(\ubar{b}^*(y),y)}{e_3(\ubar{b}^*(y),y)} > 0,
\end{align*}
where the positivity follows from the fact that $e_3>0$ and $e_4<0$ thanks to Lemma \ref{lemma_c_sign}. 
This allows us to take logarithms on both sides and solve this equation explicitly for $\bar{b}^*(y)$ -- in terms of $\ubar{b}^*(y)$ -- thanks to the expressions \eqref{def:delta} of $\psi_r$ and $\phi_r$. 
This then yields the desired expression \eqref{upperb_solv}.

\vspace{1mm}
{\it Step 3}. 
Substituting the expression \eqref{upperb_solv} of $\bar{b}^*(y)$ back in the first equation of \eqref{eq_system_midy} gives the equation $H(\ubar{b}^*(y),y) = 1$, where $H$ is defined by \eqref{H_eq_midy}. 
In order to show the existence of a unique solution $\ubar{b}^*(y) \in (b^*_{r+q},y)$ to the equations $H(\ubar{b}^*(y),y)=1$, we first note that $H(\cdot,\cdot)$ is continuous and well defined on the domain $\{(b,y) : b \in [b^*_{r+q},y], y\in[y_l,y_u]\}$, thanks to Lemma \ref{lemma_steps2dim}.
Also note that the domain of $H$ is always non-empty since we have $b^*_{r+q}<y_l$ by Lemma \ref{lemma_yl}.

Then, by using Lemma \ref{lemma_steps2dim}.(i) and (iii), we see that $H(b^*_{r+q},y)>1$ for all $y\in (y_l,y_u)$, while by using Lemma \ref{lemma_steps2dim}.(ii) and (iv), we see that $H(y,y)<1$ for all $y\in (y_l,y_u)$. 
Hence, it follows by the intermediate value theorem, that for any $y\in (y_l,y_u)$, there exists $\ubar{b}^*(y) \in (b^*_{r+q},y)$ such that $H(\ubar{b}^*(y),y)=1$. 
By rewriting the expression \eqref{H_eq_midy} in the form
\begin{align} \label{H_eq_midy2}
H(b,y)= e_3(b,y)\left(-\frac{e_4(b,y)}{ e_3(b,y)}\right)^{\frac{\gamma_1(r)}{\gamma_1(r)-\gamma_2(r)}}\left(-\frac{\gamma_2(r)}{\gamma_1(r)}\right)^{\frac{\gamma_1(r)+\gamma_2(r)}{\gamma_1(r)-\gamma_2(r)}}(\gamma_1(r)-\gamma_2(r)),
\end{align}
we observe from Lemma \ref{lemma_steps2dim_2}.(i)--(ii) that $H(\cdot,y)$ is strictly decreasing on $(b^*_{r+q},y)$. 
This implies that the solution $\ubar{b}^*(y)$ to the equation $H(\ubar{b}^*(y),y)=1$ is unique on $(b^*_{r+q}, y)$. 

\vspace{1mm}
{\it Step 4}.
Combining the result in Step 3 with the expression \eqref{upperb_solv} obtained in Step 2, yields the existence of a unique pair $(\ubar{b}^*(y),\bar{b}^*(y))$ satisfying $\ubar{b}^*(y) \in (b^*_{r+q}, y)$ and the system of equations in \eqref{eq_system_midy}. 
What remains to be shown is that $\bar{b}^*(y) \in (y,\infty)$. 

To that end, we recall from Lemma \ref{lemma_steps2dim_2}.(i) that 
\begin{align*}
\bar{b}^*(y)-y=\frac{1}{\gamma_1(r)-\gamma_2(r)}\log\left(-\frac{\gamma_2^2(r)e_4(\ubar{b}^*(y),y)}{\gamma_1^2(r)e_3(\ubar{b}^*(y),y)}\right) - y 
> \frac{1}{\gamma_1(r)-\gamma_2(r)} \log\left(-\frac{\gamma_2^2(r)e_4(y,y)}{\gamma_1^2(r)e_3(y,y)}\right) - y = L(y),
\end{align*}
where $L$ is defined in Lemma \ref{lemma_steps2dim_2} and satisfies $L(y)>0$ for all $y\in (y_l,y_u)$, thanks to Lemma \ref{lemma_steps2dim_2}.(iii)--(iv). 
This implies that $\bar{b}^*(y)>y$ for all $y\in (y_l,y_u)$ and concludes the proof.
\end{proof}

\begin{proof}[Proof of Proposition \ref{prop_transition}] 
We first note from Lemma \ref{lemma_steps2dim} that the function $H(b,y)$ defined by \eqref{H_eq_midy} is continuous for all $(b,y) \in [b^*_{r+q},y] \times [y_l,y_u]$. 
Also, for any fixed $y \in [y_l, y_u]$, by rewriting the expression \eqref{H_eq_midy} in the form
\begin{align*} 
H(b,y)= (\gamma_1(r)-\gamma_2(r)) \left(-\frac{\gamma_2(r)}{\gamma_1(r)}\right)^{\frac{\gamma_1(r)+\gamma_2(r)}{\gamma_1(r)-\gamma_2(r)}}  e_3(b,y) \left(-\frac{e_4(b,y)}{ e_3(b,y)}\right)^{\frac{\gamma_1(r)}{\gamma_1(r)-\gamma_2(r)}},
\end{align*}
we observe from Lemma \ref{lemma_steps2dim_2}.(i)--(ii) that $H(\cdot,y)$ is strictly decreasing on $(b^*_{r+q},y)$. 
Using this together with Lemma \ref{lemma_steps2dim}.(i)--(ii), we see that
\begin{enumerate}
\item [(a)] $H(y_u, y_u) = 1$ and there exists no $b \in [b^*_{r+q}, y_u)$, such that $H(b, y_u) = 1$.  
\item [(b)] $H(b_{r+q}^*, y_l) = 1$ and there exists no $b \in (b_{r+q}^*, y_l]$, such that $H(b, y_l) = 1$.  
\end{enumerate}

{\it Proof of part} (i). 
We prove separately each limit.

{\it Proof of $\ubar{b}^*(y)-y \to 0$ as $y\uparrow y_u$}.
Pick a sequence $y^{(k)} \uparrow y_u$ as $k \to \infty$, such that $y^{(k)} \neq y_u$ for all $k \in \N$, and note that by \eqref{H_eq_midy}, we can construct a sequence $\ubar{b}^*(y^{(k)})$ such that $H(\ubar{b}^*(y^{(k)}), y^{(k)}) = 1$ (by definition), $\ubar{b}^*(y^{(k)})< y^{(k)}$ and is bounded in $[0, y_u]$, according to Lemma \ref{fbp2_ex}. 
Thus, by the continuity of $H$, any accumulation point $\hat b$ of the sequence $\ubar{b}^*(y^{(k)})$ must also satisfy $H(\hat b, y_u) = 1$. 
However, since $b=y_u$ is the only point in $[b^*_{r+q}, y_u]$ satisfying $H(b, y_u) = 1$ (see (a) above), it follows that $y_u$ is the only accumulation point of the bounded sequence $\ubar{b}^*(y^{(k)})$ and hence its limit. 
This implies that $\ubar{b}^*(y^{(k)}) \to y_u$, and by $\ubar{b}^*(y^{(k)}) < y^{(k)}$ for all $k \in \N$, we conclude that $\ubar{b}^*(y)-y \to 0$ as $y\uparrow y_u$. 

{\it Proof of $\bar{b}^*(y) - y \to 0$ as $y\uparrow y_u$}.
This follows directly from the previous result, definition \eqref{upperb_solv} and Lemma \ref{lemma_steps2dim_2}.(iv), which imply that 
\begin{align*}
\lim_{y\uparrow y_u} \big\{\bar{b}^*(y) - y\big\} 
&= \frac{1}{\gamma_1(r)-\gamma_2(r)} \log\bigg(-\frac{\gamma_2^2(r) \big(\lim_{y\uparrow y_u} e_4(\ubar{b}^*(y),y)\big)}{\gamma_1^2(r) \big(\lim_{y\uparrow y_u} e_3(\ubar{b}^*(y),y) \big)}\bigg) - y_u \\
&= \frac{1}{\gamma_1(r)-\gamma_2(r)} \log\bigg(-\frac{\gamma_2^2(r) e_4(y_u,y_u)}{\gamma_1^2(r) e_3(y_u,y_u)} \bigg) - y_u = L(y_u) = 0. 
\end{align*}

\vspace{1mm}
{\it Proof of part} (ii).
We prove each limit separately. 

{\it Proof of $\ubar{b}^*(y) \to b_{r+q}^*$ as $y\downarrow y_l$}.
This can be proved following similar arguments as in the limit of $\ubar{b}^*(y)$ in part (i), by using the statement (b) instead of (a).

{\it Proof of $\bar{b}^*(y) \to b^*(y_l)$ as $y\downarrow y_l$}.
Using its definition \eqref{upperb_solv} and that $\ubar{b}^*(y) \to b_{r+q}^*$ as $y\downarrow y_l$, we get
\begin{align*}
&\lim_{y\downarrow y_l} \bar{b}^*(y) 
= \frac{1}{\gamma_1(r)-\gamma_2(r)} \log\bigg(-\frac{\gamma_2^2(r) \big(\lim_{y\downarrow y_l} e_4(\ubar{b}^*(y),y)\big)}{\gamma_1^2(r) \big(\lim_{y\downarrow y_l} e_3(\ubar{b}^*(y),y) \big)}\bigg) 
= \frac{1}{\gamma_1(r)-\gamma_2(r)} \log\bigg(-\frac{\gamma_2^2(r) e_4(b_{r+q}^*,y_l)}{\gamma_1^2(r) e_3(b_{r+q}^*,y_l)} \bigg) \\
&= y_l + \frac{1}{\gamma_1(r)-\gamma_2(r)} \log\bigg(\frac{\gamma_2^2(r)}{\gamma_1^2(r)} \, \frac{e^{\gamma_1(r+q)y_l}\bigl(\gamma_1(r+q) - \gamma_1(r)\bigr) + e^{\gamma_2(r+q)y_l}\bigl(\gamma_1(r) - \gamma_2(r+q)\bigr)}{e^{\gamma_1(r+q)y_l}\bigl(\gamma_1(r+q) - \gamma_2(r)\bigr) + e^{\gamma_2(r+q)y_l}\bigl(\gamma_2(r) - \gamma_2(r + q)\bigr)} \bigg) 
= b^*(y_l)
\end{align*}
where the penultimate equality follows similarly to the proof of Lemma \ref{lemma_steps2dim}, and the latter one from \eqref{def_deltay}.
\end{proof}

\begin{lemma} \label{lemma_c_sign}
The functions $e_3$ and $e_4$ from \eqref{constants_complex} satisfy $e_3(b,y)>0$ and $e_4(b,y)<0$ for all $(b,y) \in [b_{r+q}^*,y] \times [y_l,y_u]$.
\end{lemma}
\begin{proof}
We treat the following two cases separately. 

{\it Case $b^*_{r+q} \leq b <y$ for $e_3$}. 
We observe from the definition of $e_3(b,y)$ in \eqref{constants_complex} and the expression of $w(y;y)$ in \eqref{w_sol_double_def}
that 
$$
e_3(b,y) = \frac{\phi_{r}(y) w'(y;y) -\phi'_{r}(y) w(y;y)}{\psi'_{r}(y) \phi_{r}(y) -\psi_{r}(y) \phi'_{r}(y)}.
$$
Then, we observe that $w(\cdot;y)$ satisfies \eqref{FBP_medya} and solves the ODE \eqref{FBP_medyc} on $(b,y)$, hence it follows from Lemma \ref{Lemma_ODE_Sol}.(i) with $\rho=r+q$, $\ubar{x} = b$, $f(\ubar{x}) = V_{r+q}(b) > 0$ and $f'(\ubar{x}) = V_{r+q}'(b) = 1$, that $w(y;y), w'(y;y) > 0$ (where we also used the fact that $w(\cdot;y)$ is $C^1(\{y\})$). 
Combining this with the expressions \eqref{def:delta} of $\psi_r$ and $\phi_r$, we conclude that $e_3(b,y)$ is strictly positive in this case. 

{\it Case $b=y$ for $e_3$}. 
We can directly compute from its definition and $\gamma_2(r) < 0 < \gamma_1(r)$ that
\begin{align*}
e_3(y,y)&=e^{-\gamma_1(r)y}\frac{1-\gamma_2(r)V_{r+q}(y)}{\gamma_1(r)-\gamma_2(r)}>0.
\end{align*}

{\it Case $b^*_{r+q} \leq b \leq y$ for $e_4$}. 
Substituting the expressions of $e_1(b)$ and $e_2(b)$ from \eqref{constants_complex} in the definition of $e_4(b,y)$, as well as the functions $\phi_r,\psi_r,\phi_{r+q},\psi_{r+q}$ in the numerator, we observe that the latter is given by
\begin{align*}
&e^{\gamma_2(r+q)b+(\gamma_1(r)+\gamma_1(r+q))y}\, \big(1-\gamma_2(r+q)V_{r+q}(b) \big) \big( \gamma_1(r)-\gamma_1(r+q) \big)\\
&+ e^{\gamma_1(r+q)b+(\gamma_1(r)+\gamma_2(r+q))y} \big( \gamma_1(r+q)V_{r+q}(b)-1 \big) \big(\gamma_1(r)-\gamma_2(r+q) \big) 
=e^{\gamma_1(r+q)b+(\gamma_1(r)+\gamma_2(r+q))y} \, F(b) ,
\end{align*}
where 
\begin{align*}
F(b) 
&:= e^{(y-b)(\gamma_1(r+q) - \gamma_2(r+q))} \, \big( 1-\gamma_2(r+q)V_{r+q}(b) \big) \big( \gamma_1(r)-\gamma_1(r+q) \big) \\ 
&\quad + \big( \gamma_1(r+q) V_{r+q}(b) - 1 \big) \big( \gamma_1(r)-\gamma_2(r+q) \big) \\
&\leq \big( 1-\gamma_2(r+q)V_{r+q}(b) \big) \big(\gamma_1(r)-\gamma_1(r+q) \big) 
+ \big(\gamma_1(r+q)V_{r+q}(b)-1 \big) \big(\gamma_1(r)-\gamma_2(r+q) \big) \\
&= \big(\gamma_1(r) V_{r+q}(b) - 1 \big) \big(\gamma_1(r+q)-\gamma_2(r+q) \big)
< 0,
\end{align*}
where the former inequality follows from $b \leq y$, $\gamma_2(r+q) < 0 < \gamma_1(r+q)$ and $\gamma_1(r)-\gamma_1(r+q)<0$ (cf.~\eqref{gamma_def}), while the latter inequality follows from $\gamma_1(r)V_{r+q}(\ubar{b})-1<0$. 
To see this, observe that 
we have $b \geq b^*_{r+q}$, hence we have from \eqref{FBP_single_tresholdb}, \eqref{lemma_useful_results_Vb}, $b \leq y \leq y_u$, and \eqref{yupper_def}, that 
$$
V_{r+q}(b) 
= V_{r+q}(b^*_{r+q}) + b - b^*_{r+q}
= \frac{\mu}{r+q} + b - b^*_{r+q} 
\leq \frac{\mu}{r+q} + y_u - b^*_{r+q} 
= \frac{\mu}r ,   
$$
which implies in view of \eqref{lemma_useful_results_gamma} that \begin{align}\label{V<mur}
\gamma_1(r)V_{r+q}(\ubar{b})-1 \leq \gamma_1(r)\frac{\mu}{r}-1=\frac{\gamma_1(r)}{\gamma_2(r)}<0.
\end{align}
This completes the proof. 
\end{proof}

\begin{lemma} \label{lemma_steps2dim} 
The function $H(b,y)$ defined by \eqref{H_eq_midy} is well-defined and continuous for all $(b,y) \in [b^*_{r+q},y] \times [y_l,y_u]$, and 
satisfies the following properties:
\begin{enumerate}[\rm (i)]
\item $H(b^*_{r+q},y_l)=1$;
\item $H(y_u,y_u)=1$;
\item $y \mapsto H(b^*_{r+q},y)$ is strictly increasing for all $y \in (y_l,y_u)$;
\item $y \mapsto H(y,y)$ is strictly increasing for all $y \in (y_l,y_u)$.
\end{enumerate} 
\end{lemma}

\begin{proof}
We firstly observe from Lemma \ref{lemma_c_sign} that $e_3(b,y)>0$ and $e_4(b,y)<0$ for all $b \in [b_{r+q}^*,y]$ $y\in [y_l,y_u]$, hence the function $H(\ubar{b},y)$ is continuous and well defined.
In what follows we then prove each part separately.
\vspace{1mm}
{\it Proof of part} (i).
Given the definition of $b^*_{r+q}$ in \eqref{value_classical}, which implies that $$\gamma_1^2({r+q}) \, \psi_{r+q}(b_{r+q}^*) = \gamma_2^2({r+q}) \, \phi_{r+q}(b_{r+q}^*)$$ combined with the continuity of $e_1(\cdot)$ and \eqref{lemma_useful_results_gamma}--\eqref{lemma_useful_results_Vb}, we have 
\begin{align} \label{e1br+q}
e_1(b^*_{r+q}) 
&= \frac{1 - \gamma_2(r+q) \frac{\mu}{r+q}}{\gamma_1(r+q) \psi_{r+q}(b^*_{r+q}) - \gamma_2(r+q) \psi_{r+q}(b^*_{r+q})} 
= \frac{1}{\gamma_1(r+q) \psi_{r+q}(b^*_{r+q}) - \gamma_2(r+q) \phi_{r+q}(b^*_{r+q})}
\end{align}
where the latter is the constant in the definition \eqref{value_classical} of $V_{r+q}$, and similarly $e_2(b^*_{r+q}) = - e_1(b^*_{r+q})$.
We also observe from Corollary \ref{help_corr} that $K_1(y_l) = - K_2(y_l) = e_1(b^*_{r+q}) = - e_2(b^*_{r+q})$ and $w(x;y_l) = V_{r+q}(x)$ for all $x \in [0,b^*_{r+q}]$. Combining all of these, we notice that the FBP \eqref{FBP_lowya}--\eqref{FBP_lowyd} with $y=y_l$ is equivalent to the FBP \eqref{FBP_medya}--\eqref{FBP_medye} with fixed choices $y=y_l$ and $\ubar{b}^*(y_l)=b^*_{r+q}$. 
This implies that $e_3(b^*_{r+q},y_l) = K_3(y_l)$, $e_4(b^*_{r+q},y_l) = K_4(y_l)$, and $\bar{b}^*(y_l) = b^*(y_l)$, where $b^*$ is defined by \eqref{def_deltay}. 
As a consequence, we get from the expression of $H$ in \eqref{H_eq_midy},  $e_3$ and $e_4$ from \eqref{constants_complex} and $\bar{b}^*$ from \eqref{upperb_solv} that
\begin{align*}
H(\ubar{b}^*(y_l),y_l)
&=e_3\big(\ubar{b}^*(y_l),y_l\big) \, \psi'_r\big(\bar{b}^*(y_l)\big) + e_4\big(\ubar{b}^*(y_l),y_l\big) \, \phi'_r\big(\bar{b}^*(y_l)\big) \\
&=e_3(b^*_{r+q},y_l) \, \psi'_r(\bar{b}^*(y_l)) + e_4(b^*_{r+q},y_l) \, \phi'_r(\bar{b}^*( y_l)) =K_3(y_l) \, \psi'_r\left(b^*(y_l)\right) + K_4(y_l) \, \phi'_r\left(b^*(y_l)\right)=1,
\end{align*}
where the last equality is due to $w'(b^*(y_l);y_l)=1$ from Lemma \ref{sol:FBP1}.

\vspace{1mm} 
{\it Proof of part} (ii). 
Take an arbitrary value $y \in (y_l,y_u)$ and observe from their definitions in \eqref{constants_complex} that 
\begin{align} \label{c_3c_4yy}
\begin{split}
e_3(y,y) 
&= \frac{(\phi_{r}-\phi'_{r}V_{r+q})(y)}{(\psi'_{r}\phi_{r}-\psi_{r}\phi'_{r})(y)} 
= \frac{1-\gamma_2(r)V_{r+q}(y)}{\gamma_1(r)-\gamma_2(r)} \, e^{-\gamma_1(r)y} > 0,\\ 
e_4(y,y)
&= \frac{(\psi'_{r}V_{r+q}-\psi_{r})(y)}{(\psi'_{r}\phi_{r}-\psi_{r}\phi'_{r})(y)} = \frac{\gamma_1(r)V_{r+q}(y)-1}{\gamma_1(r)-\gamma_2(r)} \, e^{-\gamma_2(r)y} < 0 ,
\end{split}
\end{align}
where the latter inequality holds true for all $y \in [y_l,y_u]$, since for $b^*_{r+q} < y_l \leq y \leq y_u$, we have from \eqref{FBP_single_tresholdb}, \eqref{lemma_useful_results_Vb} and \eqref{yupper_def}, that 
\begin{equation} \label{V(y)<mur}
V_{r+q}(y) \leq V_{r+q}(y_u) 
= V_{r+q}(b^*_{r+q}) + \frac{\mu}r - \frac{\mu}{r+q} 
= \frac{\mu}r < \frac{1}{\gamma_1(r)} . \end{equation}
Substituting these expressions into the definition \eqref{H_eq_midy} of $H$ yields 
\begin{align} \label{hyy}
\hspace{-3mm}
H(y,y) 
= \big( 1-\gamma_2(r)V_{r+q}(y) \big)^{\frac{-\gamma_2(r)}{\gamma_1(r)-\gamma_2(r)}} \big(1-\gamma_1(r)V_{r+q}(y) \big)^{\frac{\gamma_1(r)}{\gamma_1(r)-\gamma_2(r)}} \Big( -\frac{\gamma_2(r)}{\gamma_1(r)} \Big)^{\frac{\gamma_1(r)+\gamma_2(r)}{\gamma_1(r)-\gamma_2(r)}}.
\end{align}
Hence, using \eqref{lemma_useful_results_gamma} and \eqref{V(y)<mur}, we get that 
\begin{align*}
H(y_u,y_u) 
&= \left(1-\gamma_2(r)\frac{\mu}{r}\right)^{\frac{-\gamma_2(r)}{\gamma_1(r)-\gamma_2(r)}} \left(1-\gamma_1(r)\frac{\mu}{r}\right)^{\frac{\gamma_1(r)}{\gamma_1(r)-\gamma_2(r)}} \Big(-\frac{\gamma_2(r)}{\gamma_1(r)}\Big)^{\frac{\gamma_1(r)+\gamma_2(r)}{\gamma_1(r)-\gamma_2(r)}}\\
&= \Big(-\frac{\gamma_2(r}{\gamma_1(r)}\Big)^{\frac{-\gamma_2(r)}{\gamma_1(r)-\gamma_2(r)}} \Big(-\frac{\gamma_1(r)}{\gamma_2(r)}\Big)^{\frac{\gamma_1(r)}{\gamma_1(r)-\gamma_2(r)}} \Big(-\frac{\gamma_2(r)}{\gamma_1(r)}\Big)^{\frac{\gamma_1(r)+\gamma_2(r)}{\gamma_1(r)-\gamma_2(r)}} 
= 1.
\end{align*}

\vspace{1mm} 
{\it Proof of part} (iii). 
Substituting the expressions of $e_1(b^*_{r+q}) = - e_2(b^*_{r+q})$ from part (i) (cf.~\eqref{e1br+q}) into the expressions of $e_3$ and $e_4$ from \eqref{constants_complex}, we get 
\begin{align*}
e_3(b^*_{r+q},y) 
&= \frac{\big(\phi_r(\psi_{r+q}'-\phi_{r+q}')-\phi'_r(\psi_{r+q}-\phi_{r+q})\big)(y)}{(\psi'_r\phi_r-\phi'_r\psi_r)(y) \, (\psi'_{r+q}-\phi'_{r+q})(b^*_{r+q})},\\
&=e^{-\gamma_1(r)y} \frac{\gamma_1(r+q)e^{\gamma_1(r+q)y}-\gamma_2(r+q)e^{\gamma_2(r+q)y}-\gamma_2(r)\left(e^{\gamma_1(r+q)y}-e^{\gamma_2(r+q)y}\right)}{(\gamma_1(r)-\gamma_2(r)) \big(\gamma_1(r+q) e^{\gamma_1(r+q)b^*_{r+q}} - \gamma_2(r+q) e^{\gamma_2(r+q)b^*_{r+q}} \big)}\\
e_4(b^*_{r+q},y) 
&= \frac{\big(\psi'_r(\psi_{r+q}-\phi_{r+q}) - \psi_r(\psi'_{r+q}-\phi'_{r+q})\big)(y)}{(\psi'_r\phi_r-\phi'_r\psi_r)(y) \, (\psi'_{r+q}-\phi'_{r+q})(b^*_{r+q})}\\ 
&= e^{-\gamma_2(r)y} \frac{\gamma_1(r)\left(e^{\gamma_1(r+q)y}-e^{\gamma_2(r+q)y}\right) - \gamma_1(r+q) e^{\gamma_1(r+q)y} + \gamma_2(r+q)e^{\gamma_2(r+q)y}}{(\gamma_1(r)-\gamma_2(r)) \big(\gamma_1(r+q) e^{\gamma_1(r+q)b^*_{r+q}} - \gamma_2(r+q) e^{\gamma_2(r+q)b^*_{r+q}}\big)}.
\end{align*}
Then by differentiation, we get 
\begin{align*}
\frac{d}{dy} e_3(b^*_{r+q},y) 
&= -\frac{A \, e^{(\gamma_1(r+q)-\gamma_1(r))y} + B \, e^{(\gamma_2(r+q)-\gamma_1(r))y}}{C(b^*_{r+q})}, \\
\frac{d}{dy} e_4(b^*_{r+q},y) 
&= \frac{A \, e^{(\gamma_1(r+q)-\gamma_2(r))y} + B \, e^{(\gamma_2(r+q)-\gamma_2(r))y}}{C(b^*_{r+q})},
\end{align*}
where (note that the signs follow from \eqref{gamma_def})
\begin{align}\label{AB_defmidy}
\begin{split}
A &:= (\gamma_1(r+q)-\gamma_2(r))(\gamma_1(r)-\gamma_1(r+q)) < 0 ,\\ 
B &:= (\gamma_2(r+q)-\gamma_2(r))(\gamma_2(r+q)-\gamma_1(r)) > 0 ,\\
C(b^*_{r+q}) &:= (\gamma_1(r)-\gamma_2(r)) \big(\gamma_1(r+q) e^{\gamma_1(r+q)b^*_{r+q}} - \gamma_2(r+q) e^{\gamma_2(r+q)b^*_{r+q}} \big) > 0 .
\end{split}
\end{align}
Using the relationships in \eqref{lemma_useful_results_gamma}, we can further show that $A = -B$,
which implies that 
\begin{align} \label{de34}
\begin{split}
\frac{d}{dy} e_3(b^*_{r+q},y) 
&= \frac{B \big(e^{(\gamma_1(r+q)-\gamma_1(r))y} - e^{(\gamma_2(r+q)-\gamma_1(r))y}\big)}{C(b^*_{r+q})}
>0, \quad  y \in (y_l,y_u), \\
\frac{d}{dy} e_4(b^*_{r+q},y) 
&= \frac{A \big(e^{(\gamma_1(r+q)-\gamma_2(r))y} - e^{(\gamma_2(r+q)-\gamma_2(r))y}\big)}{C(b^*_{r+q})}
<0, \quad  y \in (y_l,y_u),
\end{split}
\end{align}
thanks to $\gamma_1(r+q)-\gamma_2(r)>\gamma_1(r+q)-\gamma_1(r)>0$ and $\gamma_2(r+q)-\gamma_1(r)<\gamma_2(r+q)-\gamma_2(r)<0$. 
Combining the inequalities in \eqref{de34} with the expression of $H(b^*_{r+q},y)$ in \eqref{H_eq_midy} yileds the desired result. 

\vspace{1mm} 
{\it Proof of part} (iv).
In order to obtain the monotonicity of $y \mapsto H(y,y)$ we derive the expression \eqref{hyy} and use the fact that $V'_{r+q}(y)=1$ for all $y > y_l > b^*_{r+q}$ thanks to \eqref{FBP_single_tresholdb}, and obtain 
\begin{align*}
\frac{d}{dy} H(y,y) 
&= H(y,y) \left(\frac{\gamma_2^2(r)}{(\gamma_1(r)-\gamma_2(r))(1-\gamma_2(r)V_{r+q}(y))}-\frac{\gamma_1^2(r)}{(\gamma_1(r)-\gamma_2(r))(1-\gamma_1(r)V_{r+q}(y))}\right)\\
&=  H(y,y) \, \frac{\gamma_1(r) \gamma_2(r) \big(\gamma_1(r)-\gamma_2(r) \big) V_{r+q}(y) + \gamma_2^2(r) - \gamma_1^2(r)}{\big(\gamma_1(r)-\gamma_2(r) \big) \big(1 - \gamma_2(r) V_{r+q}(y) \big) \big(1 - \gamma_1(r) V_{r+q}(y) \big)} \\
&= H(y,y) \, \frac{-\gamma_2(r)\left(1 - \gamma_1(r)V_{r+q}(y)\right) -\gamma_1(r)}{(1-\gamma_2(r)V_{r+q}(y))(1-\gamma_1(r)V_{r+q}(y))} > 0,
\end{align*}
where the positivity follows from the expression for $H(y,y)$ in \eqref{hyy} together with \eqref{V(y)<mur} and \eqref{lemma_useful_results_gamma} which imply that $1 - \gamma_1(r)V_{r+q}(y) > 1 - \gamma_1(r)V_{r+q}(y_u) = -\gamma_1(r)/\gamma_2(r) > 0$, for all $y \in (y_l, y_u)$. 
\end{proof}

\begin{lemma} \label{lemma_steps2dim_2}
Recall the functions $e_3$ and $e_4$ from \eqref{constants_complex} and define 
$$
\Lambda(b,y):=-\frac{e_4(b,y)}{e_3(b,y)} 
\quad \text{and} \quad 
L(y):=\frac{1}{\gamma_1(r)-\gamma_2(r)} \log\Big(\frac{\gamma_2^2(r)}{\gamma_1^2(r)}\Lambda(y,y)\Big) - y, 
\quad 
$$
for all $(b,y) \in [b_{r+q}^*,y] \times [y_l,y_u]$.
Then, we have
\begin{enumerate}[\rm (i)]
\item For any $y\in [y_l,y_u]$, the function $b \mapsto \Lambda(b,y)$ is strictly decreasing on $(b_{r+q}^*,y)$;

\item For any $y\in [y_l,y_u]$, the function $b \mapsto e_3(b,y)$ is strictly decreasing on $(b_{r+q}^*,y)$; 

\item The function $y \mapsto L(y)$ is strictly decreasing on $(y_l,y_u)$; 

\item $L(y_u)=0$.
\end{enumerate}
\end{lemma}

\begin{proof}
We prove each part separately.

\vspace{1mm}
{\it Proof of part} (i).
It follows from \eqref{constants_complex} that $\Lambda$ takes the form
\begin{align*}
\Lambda(b,y) 
= - \frac{(\phi_{r+q}-\phi'_{r+q}V_{r+q})(b) \, \tilde{\psi}(y) + (\psi'_{r+q}V_{r+q}-\psi_{r+q})(b) \, \delta_1(y)}{(\phi_{r+q}-\phi'_{r+q}V_{r+q})(b) \, \delta_2(y) + (\psi'_{r+q}V_{r+q}-\psi_{r+q})(b) \, \tilde{\phi}(y)} =: - \frac{N(b,y)}{M(b,y)},
\end{align*}
where $N$ and $M$ denote the numerator and denominator of the last expression, respectively, and
\begin{align}\label{help_expression}
\begin{split}
\tilde{\psi}:=\psi'_r\psi_{r+q}-\psi_r\psi'_{r+q},\qquad&
\tilde{\phi}:=\phi_r\phi'_{r+q}-\phi'_r\phi_{r+q}\\
\delta_1:=\psi'_r\phi_{r+q}-\psi_r\phi'_{r+q}, \qquad&
\delta_2:=\phi_r\psi'_{r+q}-\phi'_r\psi_{r+q}.
\end{split}
\end{align}
Taking the derivative of $\Lambda$ with respect to $b$, we get 
\begin{align*}
M(b,y)^2 \, \frac{d}{db} \Lambda(b,y) 
&= -M(b,y) \frac{d}{db} N(b,y) + N(b,y) \frac{d}{db} M(b,y) \\
&= (\tilde{\psi}\tilde{\phi}-\delta_1\delta_2)(y) \, V_{r+q}(b) \left(\phi''_{r+q}(\psi'_{r+q}V_{r+q}-\psi_{r+q})+\psi''_{r+q}(\phi_{r+q}-\phi'_{r+q}V_{r+q})\right)(b),
\end{align*}
where in the second equality we used the properties that $V'_{r+q}(b)=1$ for $b\geq b_{r+q}^*$. 
In view of \eqref{help_expression} and \eqref{gamma_def}, we have  
\begin{align*}
(\tilde{\psi}\tilde{\phi}-\delta_1\delta_2)(y)=
-(\gamma_1(r+q) -\gamma_2(r+q))(\gamma_1(r)-\gamma_2(r))e^{(\gamma_1(r)+ \gamma_2(r)+\gamma_1(r+q) + \gamma_2(r+q))y} < 0.
\end{align*}
We also have 
\begin{align*}
&\left(\phi''_{r+q}(\psi'_{r+q}V_{r+q}-\psi_{r+q})+\psi''_{r+q}(\phi_{r+q}-\phi'_{r+q}V_{r+q})\right)(b)\\
&=e^{(\gamma_1(r+q) + \gamma_2(r+q))y}\left(\gamma_2^2(r+q)(\gamma_1(r+q)V_{r+q}(b)-1)+\gamma_1^2(r+q)(1-\gamma_2(r+q)V_{r+q}(b))\right) > 0,
\end{align*}
whose positivity results from \eqref{lemma_useful_results_Vb}, \eqref{value_classical} and 
\begin{align} \label{help_Vb} 
V_{r+q}(b)>\frac{\gamma_2^2(r+q)-\gamma_1^2(r+q)}{\gamma_2^2(r+q)\gamma_1(r+q)-\gamma_1^2(r+q)\gamma_2(r+q)}=V_{r+q}(b_{r+q}^*), \quad b > b^*_{r+q} .
\end{align}
Using the signs of these terms we conclude that $\frac{d}{db} \Lambda(b,y) <0$ for all $b > b_{r+q}^*$.

\vspace{1mm}
{\it Proof of part} (ii).
Using the notation in \eqref{help_expression}, we firstly express $e_3$ in the form
\begin{align*}
e_3(b,y) 
= \frac{(\phi_{r+q}-\phi'_{r+q}V_{r+q})(b) \, \delta_2(y) + (\psi'_{r+q}V_{r+q}-\psi_{r+q})(b) \, \tilde\phi(y)}{(\psi'_{r+q}\phi_{r+q}-\phi'_{r+q}\psi_{r+q})(b) \, (\psi'_{r}\phi_r-\phi'_{r}\psi_{r})(y)}.
\end{align*}
Then, using that $V'_{r+q}(b)=1$ for $b\geq b_{r+q}^*$, we take the derivative of $e_3$ with respect to $b$ and get
\begin{align*}
\frac{d}{db}e_3(b,y) 
&= \frac{-\delta_2(y)(\phi_{r+q}''V_{r+q}\delta_3+(\phi_{r+q}-\phi'_{r+q}V_{r+q})\delta_3')(b)+\tilde\phi(y)(\psi_{r+q}''V_{r+q}\delta_3 -(\psi'_{r+q}V_{r+q}-\psi_{r+q})\delta_3')(b)}{\delta_3(b)^2(\psi'_{r}\phi_r-\phi'_{r}\psi_{r})(y)} 
\end{align*}
where we define 
\begin{align*}
\delta_3 &:= \psi'_{r+q}\phi_{r+q}-\phi'_{r+q}\psi_{r+q}. 
\end{align*}
The denominator is straightforwardly positive, hence in what follows we focus on the terms of the numerator.
By substituting the definitions of $\psi,\phi$ from \eqref{def:delta} and $\delta_3$ from above, we have that
\begin{align*}
(\phi_{r+q}'' V_{r+q} \delta_3 + (\phi_{r+q} - \phi'_{r+q} V_{r+q}) \delta_3')(b) 
&= \phi_{r+q}^2(b) \, \psi_{r+q}(b) \, \widetilde{V}(b)\\
(\psi_{r+q}''V_{r+q}\delta_3 - (\psi'_{r+q}V_{r+q}-\psi_{r+q})\delta_3')(b) 
&= \phi_{r+q}(b) \, \psi_{r+q}^2(b) \, \widetilde{V}(b),
\end{align*}
where
\begin{align*}
\widetilde{V}(b):=\gamma_1^2(r+q)-\gamma_2^2(r+q)+\gamma_1(r+q)\gamma_2(r+q)(\gamma_2(r+q)-\gamma_1(r+q))V_{r+q}(b) > 0,
\end{align*}
whose positivity follows from \eqref{help_Vb}.
Using the above expressions, we thus have \begin{align*}
\delta_3(b)^2 \, (\psi'_{r}\phi_r-\phi'_{r}\psi_{r})(y) \, \frac{d}{db}e_3(b,y) 
= \psi_{r+q}(b) \, \phi_{r+q}(b) \, \big(\tilde\phi(y)\psi_{r+q}(b)-\delta_2(y)\phi_{r+q}(b) \big) \, \tilde{V}(b) < 0,
\end{align*}
whose negativity follows from the inequalities 
\begin{align*}
\tilde{\phi}(y) &= \phi_r(y) \, \phi_{r+q}(y) \, (\gamma_2(r+q)-\gamma_2(r)) < 0 ,\\
\delta_2(y) &= \phi_r(y) \, \psi_{r+q}(y) \, (\gamma_1(r+q) - \gamma_2(r)) > 0 .
\end{align*}
 
\vspace{1mm}
{\it Proof of part} (iii).
Using the definitions \eqref{constants_complex} of $e_3$ and $e_4$, we express $L$ in the form  
\begin{align}
L(y)
&= \frac{1}{\gamma_1(r)-\gamma_2(r)} \,  \log\left(e^{(\gamma_1(r)-\gamma_2(r))y} \, \frac{\gamma_2^2(r) \big(1-\gamma_1(r)V_{r+q}(y) \big)}{\gamma_1^2(r) \big(1-\gamma_2(r)V_{r+q}(y) \big)}\right) - y \label{L_rep}\\
&=\frac{\log\left(\gamma_2^2(r) \big(1-\gamma_1(r)V_{r+q}(y) \big)\right)-\log\left(\gamma_1^2(r) \big(1-\gamma_2(r)V_{r+q}(y) \big)\right)}{\gamma_1(r)-\gamma_2(r)}\notag
\end{align}
Using $V'_{r+q}(y)=1$, since $y>y_l>b_{r+q}^*$ by Lemma \ref{lemma_yl}, we get
\begin{align*}
L'(y)=\frac{1}{\gamma_1(r)-\gamma_2(r)}\left(\frac{-\gamma_1(r)}{1-\gamma_1(r)V_{r+q}(y)}+\frac{\gamma_2(r)}{1-\gamma_2(r)V_{r+q}(y)}\right)<0,
\end{align*}
due to \eqref{FBP_single_tresholdb}, \eqref{lemma_useful_results_gamma}, \eqref{lemma_useful_results_Vb} and \eqref{yupper_def}, which imply that $1 - \gamma_1(r)V_{r+q}(y) > 1 - \gamma_1(r)V_{r+q}(y_u) = -\gamma_1(r)/\gamma_2(r) > 0$, for all $y \in (y_l, y_u)$. 

\vspace{1mm}
{\it Proof of part} (iv). 
Recalling that $V_{r+q}(y_u)=\mu/r$ thanks to \eqref{yupper_def} and the fact that $b^*_{r+q} < y_l < y_u$, and using  \eqref{lemma_useful_results_gamma}, we immediately see that 
$\gamma_2^2(r) \, (1-\gamma_1(r)V_{r+q}(y_u)) = \gamma_1^2(r) \, (1-\gamma_2(r)V_{r+q}(y_u))$. 
Substituting this into \eqref{L_rep} yields that $L(y_u)=0$.
\end{proof}

\section{Technical results}
\label{sec:auxiliary}

In this Section, we include useful technical results that are used throughout the paper.

\begin{lemma}\label{Lemma_ODE_Sol}
Let $\ubar{x} \geq 0$ and suppose that $f\in \mathcal{C}^2(\ubar{x},\infty)$ solves the initial value problem
\begin{align*}
&\frac{\sigma^2}{2}f''(x)+\mu f'(x)-\rho f(x)=0 \quad \text{for} \quad x\in (\ubar{x},\infty) , 
\quad f(\ubar{x}) = z_0  \geq 0, 
\quad f'(\ubar{x}) = z_1 >0.
\end{align*}
Then, we have that:
\begin{enumerate}[(i)]
\item $f(\cdot)$ is strictly increasing on $(\ubar{x},\infty)$;
\item If $f$ has an inflection point $x^* \in (\ubar{x},\infty)$, then $f$ can only switch from concave on $(\ubar{x},x^*)$ to convex on $(x^*,\infty)$.
\end{enumerate}
\end{lemma}

\begin{proof}
This proof is a slight variation of a well known result \cite[Lemma 4.1.]{shreve1984optimal} (see also \cite[Lemma 2.4]{erikfinetti2023}), and is included here for completeness. 

\vspace{1mm}
{\it Proof of part} (i). 
Define  
$$
\tilde{x} := \inf \{ x \in (\ubar{x},\infty) \,|\, f'(x)=0\} > 0,
$$
where its positivity follows from $f'(\ubar x)>0$. 
Assume (aiming for a contradiction) that $\tilde{x} < \infty$. 
Then, we have $f'(x)>0$ for $x\in (\ubar x,\tilde x)$, which implies together with $f(\ubar{x})\geq 0$ that $f(\tilde x) > 0$. 
Hence, it follows from the ODE that
$\tfrac{\sigma^2}{2}f''(\tilde x)=\rho f(\tilde x)>0$, 
which implies that $f'(\tilde x-\epsilon)<0$ for sufficiently small $\epsilon > 0$. 
This is a contradiction to $f'(x)>0$ for $x\in (\ubar x,\tilde x)$, thus $\tilde{x}=\infty$ and consequently  $f'(x)>0$ for $x\in (\ubar x,\infty)$. 

\vspace{1mm}
{\it Proof of part} (ii). 
We first note that the unique solution to the ODE with the given boundary conditions is given by
\begin{align*}
f(x)=\frac{z_1-\gamma_2(\rho)z_0}{(\gamma_1(\rho)-\gamma_2(\rho))e^{\gamma_1(\rho)\ubar{x}}}e^{\gamma_1(\rho)x}-\frac{z_1-\gamma_1(\rho)z_0}{(\gamma_1(\rho)-\gamma_2(\rho))e^{\gamma_2(\rho)\ubar{x}}}e^{\gamma_2(\rho)x},
\quad \text{for $\gamma_1(\rho), \gamma_2(\rho)$ defined in \eqref{gamma_def},}
\end{align*}
and observe that $f\in\mathcal{C}^3(\ubar x,\infty)$. 
Suppose that $\tilde x \in (\ubar{x},\infty)$ is any point satisfying $f''(\tilde x)=0$. 
Then, taking the derivative of the ODE and using part (i), we get 
$
\tfrac{\sigma^2}{2}f'''(\tilde x) = \rho f'(\tilde x)>0 ,
$
which implies that $f''(x)$ is increasing at $x=\tilde x$. 
This further implies, for sufficiently small $\epsilon > 0$, that 
$$
f''(x) \begin{cases} 
<0 , \quad x \in (\tilde x-\epsilon, \tilde x), \\
>0 , \quad x \in (\tilde x, \tilde x+\epsilon) 
\end{cases}
\quad \Leftrightarrow \quad 
f(x) \text{ is } \begin{cases} 
\text{concave for } x \in (\tilde x-\epsilon, \tilde x), \\
\text{convex for }  x \in (\tilde x, \tilde x+\epsilon). 
\end{cases}
$$
which proves the desired result.
\end{proof}
\bibliographystyle{abbrv}
\bibliography{OmegaDiv}

\end{document}